\documentclass[12pt, leqno]{amsart}

\usepackage{graphicx}
\usepackage{amssymb,amsmath,amsthm,amscd}
\usepackage{mathrsfs}
\usepackage{lscape}
\usepackage{enumerate}
\usepackage[usenames,dvipsnames]{color}
\usepackage[colorlinks=true, pdfstartview=FitV,
 linkcolor=blue,citecolor=blue,urlcolor=blue]{hyperref}
\usepackage{tikz}

\usepackage[all]{xy}
\CompileMatrices

\usepackage{xspace}
\usepackage{marginnote}
\usepackage{verbatim}
\allowdisplaybreaks[3]

\newcommand{\nc}{\newcommand}
\numberwithin{equation}{section}

\newenvironment{rouge}{\relax\color{red}}{\relax}
\newenvironment{blue}{\relax\color{blue}}{\hspace*{.5ex}\relax}
\newenvironment{jaune}{\relax\color{Orchid}}{\hspace*{.5ex}\relax}
\nc{\hs}{\hspace*}

\newcommand{\ber}{\begin{rouge}}
\newcommand{\er}{\end{rouge}}
\nc{\bjn}{\begin{jaune}}
\nc{\ejn}{\end{jaune}}
\newcommand{\beb}{\begin{blue}}
\newcommand{\eb}{\end{blue}}

\newcommand{\berm}{\begin{rouge}{}\marginnote{\fbox{\scshape\lowercase{M}}}{}}

\newcommand{\bers}{\begin{blue}{}\marginnote{\fbox{\scshape\lowercase{O}}}{}}

\theoremstyle{plain}
\newtheorem{lemma}{Lemma}[section]
\newtheorem{proposition}[lemma]{Proposition}
\newtheorem{theorem}[lemma]{Theorem}

\newtheorem{corollary}[lemma]{Corollary}
\newtheorem{conjecture}[lemma]{Conjecture}

\theoremstyle{definition}
\newtheorem{remark}[lemma]{Remark}
\newtheorem{example}[lemma]{Example}

\newtheorem{definition}[lemma]{Definition}

\newcommand{\hd}{{\operatorname{hd}}}
\newcommand{\soc}{{\operatorname{soc}}}

\newcommand{\g}{\mathfrak{g}}

\newcommand{\F}{\mathcal{F}}
\newcommand{\C}{\mathbb{C}}
\newcommand{\Z}{\mathbb{Z}}

\newcommand{\A}{\mathbb{A}}
\newcommand{\N}{\mathsf{N}}

\newcommand{\Q}{\mathcal{Q}}
\newcommand{\cmA}{\mathsf{A}}
\newcommand{\W}{\mathsf{W}_0}
\newcommand{\seteq}{\mathbin{:=}}
\newcommand{\al}{\alpha}

\newcommand{\ga}{\gamma}
\newcommand{\re}{s}

\newcommand{\WUp}{\widehat{\Upsilon}}
\newcommand{\PR}{\Phi^+}

\newcommand{\tw}{{\widetilde{w}}}

\newcommand{\redez}{{\widetilde{w}_0}}
\newcommand{\um}{\underline{m}}
\newcommand{\un}{\underline{n}}

\newcommand{\us}{\underline{s}}
\newcommand{\up}{\underline{p}}

\newcommand{\tb}{\mathtt{b}}
\newcommand{\rl}{\mathsf{Q}}
\newcommand{\wl}{\mathsf{P}}
\newcommand{\cl}{{\rm cl}}
\newcommand{\wt}{{\rm wt}}
\newcommand{\ok}{\overline{k}}

\newcommand{\ol}{\overline{l}}
\newcommand{\dist}{{\rm dist}}

\newcommand{\het}{{\rm ht}}
\newcommand{\mQ}{\mathscr{Q}}
\newcommand{\mC}{\mathscr{C}}

\newcommand{\Rep}{{\rm Rep}}

\newcommand{\soplus}{\mathop{\mbox{\normalsize$\bigoplus$}}\limits}

\newcommand{\lf}{[\hspace{-0.3ex}[}
\newcommand{\rf}{]\hspace{-0.3ex}]}

\newcommand{\ko}{\mathbf{k}}
\newcommand{\conv}{\mathbin{\mbox{\large $\circ$}}}
\newcommand{\Lto}{\longrightarrow}

\newcommand{\tens}{\mathop\otimes}
\newcommand{\Rnorm}{R^{\rm{norm}}}

\nc{\ii}{{t}}
\newcommand{\Stom}{S_{[\redez]}(\um)}
\newcommand{\VtoQm}{V^{(\ii)}_{Q}(\um)}
\newcommand{\VtomQm}{V_{\mQ}(\um)}

\newcommand{\dct}[2]{\overset{#2}{\underset{#1}{\conv}} }
\newcommand{\prt}[2]{ \left( \begin{matrix} #1 \\ #2 \end{matrix}\right) }

\newcommand{\To}[1][{\hspace{2ex}}]{\xrightarrow{\hs{.8ex}#1\hs{.8ex}}}

\newcommand{\isoto}[1][]{\mathop{\xrightarrow%
[{\raisebox{.3ex}[0ex][.3ex]{$\scriptstyle{#1}$}}]%
{{\raisebox{-.6ex}[0ex][-.6ex]{$\mspace{2mu}\sim\mspace{2mu}$}}}}}

\newlength{\mylength}
\newcommand*\ov[1]{\overline{#1}}

\nc{\lan}{\langle}
\nc{\ran}{\rangle}
\newcommand{\be}{\begin{enumerate}}
\newcommand{\ee}{\end{enumerate}}

\nc{\ro}{{\rm(}}
\nc{\rof}{{\rm)}}
\newenvironment{myequation}
{\relax\setlength{\arraycolsep}{1pt}\begin{eqnarray}}
{\end{eqnarray}}
\newenvironment{myequationn}
{\relax\setlength{\arraycolsep}{1pt}\begin{eqnarray*}}
{\end{eqnarray*}}
\nc{\eq}{\begin{myequation}}
\nc{\eneq}{\end{myequation}}
\nc{\eqn}{\begin{myequationn}}
\nc{\eneqn}{\end{myequationn}}

\newenvironment{myarray}[1]{\relax\setlength{\arraycolsep}{1pt}

\begin{array}{#1}}{\end{array}\relax}

\newcommand{\ba}{\begin{myarray}}
\newcommand{\ea}{\end{myarray}}

\nc{\bc}{\begin{cases}}
\nc{\ec}{\end{cases}}
\nc{\cor}{\ko}
\nc{\vs}{\vspace*}
\nc{\KLR}{quiver Hecke\xspace}
\nc{\fg}{\mathsf{g}}
\nc{\ffg}{ \textbf{\textit{g}}}

\title[Categorical Langlands duality]{
Categorical relations between
Langlands dual quantum affine algebras: \\ Doubly laced types}

\author[M. Kashiwara, S.-j. Oh]{Masaki Kashiwara$^\dagger$, Se-jin Oh$^\ddagger$}

\address{Research Institute for Mathematical Sciences \\
Kyoto University \\ Kyoto 606-8502, Japan \\
\& Department of Mathematical Sciences and
School of Mathematics \\
Korea Institute for Advanced Study \\ Seoul 130-722, Korea }
\email{masaki@kurims.kyoto-u.ac.jp}

\address{Department of Mathematics Ewha Womans University Seoul 120-750, Korea}
\email{sejin092@gmail.com}

\thanks{$^\dagger$ The research
was supported by Grant-in-Aid for Scientific Research (B)
15H03608, Japan Society for the Promotion of Science.}

\thanks{$^\ddagger$ This work was supported by NRF Grant \# 2016R1C1B2013135.}


\date{\today}

\begin{document}

\begin{abstract}
We prove that the Grothendieck rings of category $\mathcal{C}^{(t)}_Q$
over quantum affine algebras $U_q'(\g^{(t)})$ $(t=1,2)$
 associated to each Dynkin quiver $Q$ of finite type
$A_{2n-1}$ (resp. $D_{n+1}$)
is isomorphic to one of category $\mathcal{C}_{\mQ}$ over the Langlands dual $U_q'({^L}\g^{(2)})$ of $U_q'(\g^{(2)})$ 
associated to any twisted adapted class
$[\mQ]$ of $A_{2n-1}$ (resp. $D_{n+1}$).
This results provide partial answers of conjectures of Frenkel-Hernandez on Langlands duality for finite-dimensional
representation of quantum affine algebras.
\end{abstract}

\maketitle

\section{Introduction}
Let $U_q'(\g^{(r)})$ $(r=2,3)$ be the twisted quantum affine algebra ($\g^{(r)}=A_{2n-1}^{(2)}$,$D_{n+1}^{(2)}$,$E_{6}^{(2)}$,$D_{4}^{(3)}$)
and let $U_q'({^L}\g^{(r)})$ be its untwisted Langlands dual (${}^L\g^{(r)}=B_{n}^{(1)}$,$C_{n}^{(1)}$,$F_{4}^{(1)}$,$G_{2}^{(1)}$) whose generalized Cartan matrix is the transpose of
that of $U_q'(\g^{(r)})$. Let us denote by $\mathcal{W}_{q,t}(\ffg)$ the $(q,t)$-deformed $\mathcal{W}(\ffg)$-algebra associated to the simple Lie subalgebra $\ffg$ of $^{L}\g^{(r)}$,
introduced by Frenkel and Reshetikhin in \cite{FR98}.
Then it is known that (i) the limit $q \to {\rm exp}(\pi i /r)$  of $\mathcal{W}_{q,t}(\ffg)$ recovers the commutative Grothendieck ring of finite-dimensional integrable representations
$\mathcal{C}_{\g^{(r)}}$ over $U_q'(\g^{(r)})$, and (ii) the limit $t \to 1$ of $\mathcal{W}_{q,t}(\ffg)$ recovers the one of $\mathcal{C}_{{^L}\g^{(r)}}$ over $U_q'({^L}\g^{(r)})$ (\cite{FR99}):
\begin{align} \label{eq: known 1}
\xymatrix@R=5ex@C=15ex{ \left [\mathcal{C}_{\g^{(r)}} \right] \ar@/^1.5pc/@{<-->}[rr] & \mathcal{W}_{q,t}(\ffg) \ar@{->}[l]^{{\rm exp}(\pi i /r)  \gets q}
\ar@{->}[r]_{ t \to 1 } & \left[\mathcal{C}_{{^L}\g^{(r)}}\right]}
\end{align}
Thus $\mathcal{W}_{q,t}(\ffg)$ {\em interpolates} the
Grothendieck rings of the categories $\mathcal{C}_{\g^{(r)}}$ and $\mathcal{C}_{{^L}\g^{(r)}}$. Since then,  the duality between the representations
over $U_q'(\g^{(r)})$ and $U_q'({^L}\g^{(r)})$ has been intensively studied (for example, see \cite{FH11A,FH11B}). 
We remark that this duality is related to the geometric Langlands correspondence (see \cite[Introduction]{FH11B}).

On the other hand, Hernandez (\cite{H10}) proved that, for the untwisted quantum affine algebra $U_q'(\g^{(1)})$ ($\g^{(1)}=A_{2n-1}^{(1)}$,$D_{n+1}^{(1)}$,$E_{6}^{(1)}$,$D_{4}^{(1)}$)
corresponding to $U_q'(\g^{(r)})$, the commutative Grothendieck ring of $\mathcal{C}_{\g^{(1)}}$ is isomorphic to the one of $\mathcal{C}_{\g^{(r)}}$:
\begin{align} \label{eq: known 2}
\xymatrix@R=5ex@C=15ex{ \left [\mathcal{C}_{\g^{(r)}} \right] \ar@{<->}[rr]_{\simeq } && \left[\mathcal{C}_{\g^{(1)}}\right]}
\end{align}
Hence, we can expect that the isomorphisms among the Grothendieck groups can be lifted to equivalences of categories:

\begin{align} \label{eq: dias intro}
\raisebox{2em}{\xymatrix@C=.6ex@R=5ex{ && \mathcal{C}_{{^L}\g^{(r)}} \ar@{<-->}[dll]_\sim \ar@{<-->}[drr]^\sim \\
\mathcal{C}_{\g^{(1)}}  \ar@{<-->}[rrrr]^{\sim} && &&   \mathcal{C}_{\g^{(r)}}
}}
\end{align}
However, there is no
satisfactory answer for the reason why such dualities happen.
The goal of paper is to provide new point of view for these dualities through the categorification
theory of quantum groups.

\medskip
The quiver Hecke algebras $R_\fg$, introduced by Khovanov-Lauda \cite{KL09,KL11} and Rouquier \cite{R08} independently, categorify the negative part
$U^-_q(\fg)$ of quantum groups $U_q(\fg)$ for all symmetrizable Kac-Moody algebras $\fg$. The categorification for (dual) PBW-bases and global bases
of the integral form $U_\A^-(\fg)$ of $U^-_q(\fg)$ were developed very actively since the introduction of quiver Hecke algebras.
Among them, \cite{BKM12,Kato12,Mc12} give the categorification theory for (dual) PBW-bases of $U^-_\A(\fg)$ associated to finite simple Lie algebra $\fg$ by using
convex orders on the set of positive roots $\PR$.

On the other hand, Hernandez and Leclerc \cite{HL11} defined a subcategory $\mathcal{C}^{(1)}_Q$ of $\mathcal{C}_{\g^{(1)}}$
for quantum affine algebras $U_q'(\g^{(1)})$ of untwisted affine type $ADE$, which depends on the Auslander-Reiten quiver $\Gamma_Q$ for each Dynkin quiver $Q$ of
finite type $ADE$. They proved that $\mathcal{C}^{(1)}_Q$ categorifies $U^-_\A(\fg)^\vee|_{q=1}$, where $\fg$ is the finite simple Lie subalgebra of $\g^{(1)}$. Furthermore,
they provided the categorification theories for the upper global basis and for the dual PBW-basis {\it associated to $Q$}
via certain sets of modules in $\mathcal{C}^{(1)}_Q$.

For a quantum affine algebra $U_q'(\g)$, the first-named author and his collaborators constructed the quantum affine Schur-Weyl duality functor
$\F : \Rep(R^\Xi) \to \mathcal{C}_\g $
by observing {\it denominator formulas} $d_{V,W}(z)$ of the normalized $R$-matrices $\Rnorm_{V,W}(z)$ between {\it good} modules $V,W \in \mathcal{C}_\g$ (\cite{KKK13A,KKK13B}).
Here $R^\Xi$ is the quiver Hecke algebra determined by Schur-Weyl datum $\Xi$ which depends on the choice of good modules in $\mathcal{C}_\g$
(see \S\;\ref{subsec: SW datum} for details),
and we denote by  $\Rep(R^\Xi)$  the category of finite-dimensional
modules over $R^\Xi$.
In \cite{KKK13B}, they construct an exact functor
\begin{align}\label{eq: CQ1}
\text{$\F^{(1)}_Q : \Rep(R_\fg) \overset{\simeq}{\Lto} \mathcal{C}^{(1)}_Q$
sending simples to simples},
\end{align}
where $\fg=A_n$ or $D_n$ is a finite simple Lie subalgebra of $A_n^{(1)}$ or $D_n^{(1)}$,
and $Q$ is of type $\fg$, respectively.
Furthermore, the authors and their collaborators (\cite{KKKO14S})
defined the subcategory $\mathcal{C}^{(2)}_{Q'}$ of $\mathcal{C}_{\g^{(2)}}$ and constructed {\it twisted analogues} of \eqref{eq: CQ1}:
For any Dynkin quivers $Q$ and $Q'$ of type $A_n$ or $D_n$, we have
\begin{align}\label{eq: CQ1 CQ2}
\scalebox{0.88}{\xymatrix@R=0.5ex@C=6ex{ \mathcal{C}_{\g^{(1)}}  \supset \mathcal{C}^{(1)}_Q \ar@{<->}@/^1.5pc/[rr]^-{\simeq} &
\Rep(R_\fg) \ar@{->}[r]_{\F^{(2)}_{Q'}}^-{\simeq} \ar@{->}[l]^-{\F^{(1)}_{Q}}_{\simeq} & \mathcal{C}^{(2)}_{Q'} \subset \mathcal{C}_{\g^{(2)}}}\
\text{and}\
\scalebox{0.88}{\xymatrix@R=0.5ex@C=6ex{ [\mathcal{C}^{(1)}_Q] \ar@{<->}@/^1.5pc/[rr]^-{\simeq} &
[\Rep(R_\fg)] \simeq U_\A^-(\g)^\vee|_{q=1} \ar@{->}[r]_{\qquad \qquad \simeq } \ar@{->}[l]^{\simeq \qquad \qquad} & [\mathcal{C}^{(2)}_{Q'}]}}}
\end{align}
Here $\fg=\g=A_n$ or $D_n$.
The above result provides the categorification theoretical interpretation
of the similarity between the modules over $\mathcal{C}_{\g^{(1)}}$ and
$\mathcal{C}_{\g^{(2)}}$, described in \eqref{eq: known 2}.

\medskip
In this paper, we define certain subcategory $\mC_{\mQ}$ of $\mathcal{C}_{B_n^{(1)}}$ and $\mathcal{C}_{C_n^{(1)}}$ for any twisted adapted class $[\mQ]$ of
finite type $A_{2n-1}$ and $D_{n+1}$,
and prove that
(1) Grothendieck rings $[\mC_{\mQ}]$ are isomorphic to $[\mathcal{C}^{(t)}_Q]$ $(t=1,2)$ for each Dynkin quiver $Q$ of finite type $A_{2n-1}$
and $D_{n+1}$, respectively,
(2) there exists an exact functor between $\mC_{\mQ}$ and $\mathcal{C}^{(t)}_Q$  $(t=1.2)$ sending simples to simples.
To explain our main result, we need to introduce several notions and previous results.

Let $Q$ be a Dynkin quiver of finite type $ADE$. By the Gabriel theorem (\cite{Gab80}), it is well-known that Auslander-Reiten(AR) quiver $\Gamma_Q$ reflects the representation theory
for the path algebra $\C Q$. Moreover, the vertices of $\Gamma_Q$ can be identified with $\PR$ and the convex partial order $\prec_Q$ of $\PR$ is represented by
the paths in $\Gamma_Q$ (see \cite{B99,Gab80} for more detail). On the other hand, each commutation class $[\redez]$ of reduced expressions for the longest element
$w_0$ of a finite Weyl group determines the convex partial order $\prec_{[\redez]}$ on $\PR$. In particular, each
$\prec_Q$ coincides with the convex partial order {\it induced from} the commutation class $[Q]$ consisting of all reduced expressions {\it adapted to $Q$},
and has its unique Coxeter element $\phi_Q$. Interestingly, all commutation classes $\{ [Q] \}$ are
{\it reflection equivalent} and hence can be grouped into one {\it $r$-cluster point $\lf Q \rf$}. In \cite{OS15}, the second-named author and Suh introduced the combinatorial
AR-quiver $\Upsilon_{[\redez]}$ for every $[\redez]$ of $w_0$ for any finite type to realize the convex partial order $\prec_{[\redez]}$ and
studied the combinatorial properties of $\Upsilon_{[\redez]}$ of type $A$.

In the papers \cite{Oh14A,Oh14D,Oh15E}, the second-named author proved that important information on the representation theories
for $U_q'(A^{(1)}_n)$ and $U_q'(D^{(1)}_n)$ are encoded in the AR-quiver $\Gamma_Q$ in the following sense:
\be[(1)]
\item
 In \cite{Oh14A,Oh14D}, he proved that
the conditions for
\begin{align}
\label{eq: Dorey CP}{\rm Hom}(V(\varpi_i)_x\otimes V(\varpi_j)_y, V(\varpi_k)_z) \ne 0,
\end{align}
can be interpreted as the {\it coordinates} of  $(\alpha,\beta,\gamma)$ in some $\Gamma_Q$ where $\alpha+\beta=\gamma \in \PR$ and $\g$ is of type $A^{(1)}_{n}$ and $D^{(1)}_{n}$.
Here the conditions in \eqref{eq: Dorey CP} are referred as Dorey's rule for quantum affine algebras of type $A^{(1)}_{n}$, $D^{(1)}_{n}$, $B^{(1)}_{n}$ and $C^{(1)}_{n}$
and studied by Chari-Pressley (\cite{CP96}) by using Coxeter elements and twisted Coxeter elements.
\item By using the newly introduced notions on the sequences of positive roots in \cite{Oh15E}, he proved that we can {\it read} the denominator formulas $d_{k,l}(z)$ for
$U_q'(A_n^{(1)})$ and $U_q'(D_n^{(1)})$ from {\it any} $\Gamma_Q$.
\ee

In \cite{OS16B,OS16C}, to extend the previous results to quantum affine algebras of type $B^{(1)}_{n}$ and $C^{(1)}_{n}$, the second-named author and Suh
developed the twisted analogues by using twisted Coxeter elements $\widetilde{\phi}$ of type $A_{2n-1}$ and $D_{n+1}$ associated to Dynkin diagram automorphisms
\begin{align}
&\ba{ll}
& \xymatrix@R=0.5ex@C=4ex{ *{\circ}<3pt>  \ar@{<->}@/^1pc/[rrrrr] \ar@{-}[r]_<{1 \ \ }  &*{\circ}<3pt> \ar@{<->}@/^/[rrr]
\ar@{-}[r]_<{2 \ \ }  &   {}
\ar@{.}[r] & *{\circ}<3pt> \ar@{-}[r]_>{\,\,\ 2n-2} &*{\circ}<3pt>\ar@{-}[r]_>{\,\,\,\, 2n-1} &*{\circ}<3pt> }
\hs{3ex}\longrightarrow\hs{2ex}
\xymatrix@R=0.5ex@C=4ex{ *{\circ}<3pt>  \ar@{-}[r]_<{1 \ \ }  &*{\circ}<3pt>
\ar@{-}[r]_<{2 \ \ }  &   {}
\ar@{.}[r] & *{\circ}<3pt> \ar@{-}[r]_>{\,\,\ n-1} &*{\circ}<3pt>\ar@{=>}[r]_>{\,\,\,\, n} &*{\circ}<3pt> }, \\
&\hs{12ex}\text{$A_{2n-1}$}\hs{35ex}\text{$B_n$} \allowdisplaybreaks \\[2ex]
& \raisebox{1em}{\xymatrix@R=0.5ex@C=4ex{
& & & & & *{\circ}<3pt>\ar@{-}[dl]^<{ \  n} \ar@{<->}@/^1pc/[dd] \\
*{\circ}<3pt>   \ar@{-}[r]_<{1 \ \ }  &*{\circ}<3pt>
\ar@{-}[r]_<{2 \ \ }  &   {}
\ar@{.}[r] & *{\circ}<3pt> \ar@{-}[r]_>{\,\,\ n-1} &*{\circ}<3pt>\ar@{}[l]^>{\,\,\,\, n-2} \\
& & &  & & *{\circ}<3pt>\ar@{-}[ul]^<{\quad \ \  n+1} \\
}}
\hs{3ex}\longrightarrow\hs{2ex}
\xymatrix@R=0.5ex@C=4ex{ *{\circ}<3pt>  \ar@{-}[r]_<{1 \ \ }  &*{\circ}<3pt>
\ar@{-}[r]_<{2 \ \ }  &   {}
\ar@{.}[r] & *{\circ}<3pt> \ar@{-}[r]_>{\,\,\,\ n-1} &*{\circ}<3pt>\ar@{<=}[r]_>{\,\,\,\, n} &*{\circ}<3pt> }.\\[-1ex]
&\hs{12ex}\text{$D_{n+1}$}\hs{35ex}\text{$C_n$}
\ea\label{dia:fold}
\end{align}
They characterized the $r$-cluster point $\lf \mQ \rf$ {\it arising} from {\it any} twisted Coxeter element $\widetilde{\phi}$
in terms of {\it Coxeter composition} (see Definition \ref{def:foldable} and \cite[Observation 4.1]{OS16B}).
We call the commutation classes $[\mQ]$ in $\lf \mQ \rf$ {\it twisted adapted classes}. (Hence the notation $\mQ$
in this paper can be understood as the label of the combinatorial AR-quiver $\Upsilon_{[\mQ]}$ for each $[\mQ]$ in $\lf \mQ \rf$.)
Moreover, by assigning coordinate system
to $\Upsilon_{[\mQ]}$ and folding $\Upsilon_{[\mQ]}$, they proved that (1$'$) Dorey's rule for quantum affine algebras of type
$B^{(1)}_{n}$ and $C^{(1)}_{n}$ can be interpreted as
the {\it coordinates} of  $(\alpha,\beta,\gamma)$ in some {\it folded AR-quiver} $\widehat{\Upsilon}_{[\mQ]}$
where $(\alpha,\beta)$ is a {\it $[\mQ]$-minimal pair} of $\gamma$, (2$'$) we can {\it read} the denominator formulas $d_{k,l}(z)$ for
$U_q'(B_n^{(1)})$ and $U_q'(C_n^{(1)})$ from {\it any} $\widehat{\Upsilon}_{[\mQ]}$.

\medskip

With the previous results at hand, we first introduce the subcategory $\mC_\mQ$ for $U_q'(B_n^{(1)})$ and
$U_q'(C_n^{(1)})$ by considering the coordinates of positive roots in $\widehat{\Upsilon}_{[\mQ]}$,
where $[\mQ]$ is a twisted adapted class of type $A_{2n-1}$ and $D_{n+1}$ respectively (Definition \ref{def: CmQ}).
By considering the denominator formulas for $U_q'(B^{(1)}_{n})$
(resp. $U_q'(C^{(1)}_{n})$) and the coordinate system of
$\widehat{\Upsilon}_{[\mQ]}$, we can take the Schur-Weyl datum $\Xi$ for each twisted adapted class $[\mQ]$ yielding the exact functor
$$ \F_\mQ : \Rep(R_\fg) \Lto \mC_\mQ \subset \mC_\g$$
where $\fg$ is the finite simple Lie algebra of type $A_{2n-1}$ and $D_{n+1}$,
and $\g=B^{(1)}_{n}$ and $C^{(1)}_{n}$, respectively (Theorem \ref{thm: exact functor}).
Furthermore, by applying the correspondence between Dorey's rule and $[\mQ]$-minimal pair, we can prove that the functor
$\F_\mQ$ {\it sends simples to simples} (Theorem \ref{thm: simples to simples}). Thus we have Langlands analogues of \eqref{eq: CQ1 CQ2}:
\begin{align}\label{eq: CQ2 Langlands}
\scalebox{0.89}{\xymatrix@R=0.5ex@C=6ex{ \mathcal{C}_{\g^{(2)}}  \supset \mathcal{C}^{(2)}_Q \ar@{<->}@/^1.5pc/[rr]^-{\simeq} &
\Rep(R_\fg) \ar@{->}[r]_{\F_{\mQ}}^-{\simeq} \ar@{->}[l]^-{\F^{(2)}_{Q}}_{\simeq} & \mC_{\mQ} \subset \mathcal{C}_{{}^L\g^{(2)}}}\
\text{and}\
\scalebox{0.89}{\xymatrix@R=0.5ex@C=6ex{ [\mathcal{C}^{(2)}_Q] \ar@{<->}@/^1.5pc/[rr]^-{\simeq} &
[\Rep(R_\fg)] \simeq U_\A^-(\g)^\vee|_{q=1} \ar@{->}[r]_{\qquad \qquad \simeq } \ar@{->}[l]^{\simeq \qquad \qquad} & [\mC_{\mQ}]}}}
\end{align}
Hence we have
\begin{align}
\ba{c}\scalebox{0.9}{\xymatrix@C=.6ex@R=5ex{ &&\hs{7ex}
\mC_\mQ \subset \mathcal{C}_{{}^L\g^{(2)}}\ar@{<.>}[ddrr]\ar@{<.>}[ddll]\\
&& \Rep(R_{\fg}) \ar[u]^-{\F_{\mQ}}\ar[drr]_-{\F^{(2)}_{Q'}}\ar[dll]^-{\F^{(1)}_{Q}} \\
\mathcal{C}_{\g^{(1)}}  \supset \mathcal{C}^{(1)}_{Q} \ar@{<.>}[rrrr] && && \hs{2ex}\mathcal{C}^{(2)}_{Q'}
\subset   \mathcal{C}_{\g^{(2)}},
}} \
\scalebox{0.9}{\xymatrix@C=6ex@R=5ex{ & [\mC_\mQ] \ar@{<->}[ddr]^{\simeq}\ar@{<->}[ddl]_{\simeq}  \\
&  U^-_\A(\fg)^\vee|_{q=1} \ar@{<->}[u]^{\simeq}\ar@{<->}[dr]_{\simeq}\ar@{<->}[dl]^{\simeq} \\
[\mathcal{C}^{(1)}_{Q}] \ar@{<->}[rr]_{\simeq} &&  [\mathcal{C}^{(2)}_{Q'}].}}
\ea
\end{align}
This result implies partial answers for conjectures suggested by Frenkel-Hernandez in the following sense:
\begin{eqnarray*} &&
\parbox{85ex}{
For a representation $V$ in $\mathcal{C}^{(2)}_{Q}$, it has a Langlands dual representation ${^L}V$ in $\mC_\mQ$ via induced functor $\F_\mQ \circ {F^{(2)}_Q}^{-1}$
for any $[Q]$ and $[\mQ]$. In particular if $V$ is simple, so is ${^L}V$
}
\end{eqnarray*}
(see \cite[Conjecture 2.2, Conjecture 2.4, Conjecture 3.10]{FH11A}).

As an application, we can characterize the sets of modules in $\mC_\mQ$ categorifying the upper global basis and the dual PBW-basis associated to $[\mQ]$ of $U_\A(\fg)^\vee|_{q=1}$
(Corollary \ref{cor: PBW upper global}). In Section \ref{sec: Simple head and socle}, we continue the study of \cite{Oh15E} about the intersection
of the dual PBW-basis $P_{[\redez]}$ associated to a commutation class $[\redez]$ and the upper global basis $B(\infty)$. More precisely, in \cite[Corollary 5.24]{Oh15E},
the second-named author proved
that an element $\mathbf{b} \in P_{[Q]} \cap B(\infty)$ if and only if $\mathbf{b}$ corresponds to a $[Q]$-simple sequence $\um$, for any $[Q] \in \lf Q\rf$.
In this paper, we also prove that
\begin{eqnarray*} &&
\parbox{85ex}{
an element $\mathbf{b} \in P_{[\mQ]} \cap B(\infty)$ if and only if $\mathbf{b}$ corresponds to a $[\mQ]$-simple sequence $\um$, for any $[\mQ] \in \lf \mQ\rf$
(Corollary \ref{cor: simple simple}).}
\end{eqnarray*}

Now we suggest the following conjecture:

\begin{conjecture}
An element $\mathbf{b} \in P_{[\redez]} \cap B(\infty)$
if and only if $\mathbf{b}$ corresponds to a $[\redez]$-simple sequence $\um$, for any $[\redez]$ of $w_0$.
\end{conjecture}

In Appendix, we propose several conjectures on the dualities among quantum affine algebras $U_q'(E^{(i)}_6)$ $(i=1,2)$ and
$U_q'(F^{(1)}_4)$ by applying the same framework of this paper. 

\section{Cluster points and their AR-quivers with coordinates}

Let $I$ be an index set. A \emph{symmetrizable Cartan datum} is a quintuple $(\cmA,\wl,\Pi,\wl^{\vee},\Pi^{\vee})$ consisting of
{\rm (a)} a \emph{symmetrizable generalized Cartan matrix} $\cmA=(a_{ij})_{i,j \in I}$,
{\rm (b)} a free abelian group $\wl$, called the \emph{weight lattice},
{\rm (c)} $\Pi= \{ \alpha_i \in \wl \mid \ i \in I \}$, called
the set of \emph{simple roots},
{\rm (d)} $\wl^{\vee}\seteq {\rm Hom}(\wl, \Z)$, called the \emph{coweight lattice},
{\rm (e)} $\Pi^{\vee}= \{ h_i \ | \ i \in I \}\subset P^{\vee}$, called
the set of \emph{simple coroots}.
It satisfies certain conditions$\colon$
$\lan h_i,\al_j\ran=a_{ij}$, etc (see \cite[\S 1.1]{KKKOIII} for precise definition).

The free abelian group $\rl\seteq\oplus_{i \in I} \Z \alpha_i$ is called the
\emph{root lattice}. Set $\rl^{+}= \sum_{i \in I} \Z_{\ge 0}
\alpha_i$. For $\mathsf{b}=\sum_{i\in I}m_i\al_i\in\rl^+$,
we set $\het(\mathsf{b})=\sum_{i\in I}m_i$.

We denote by $U_q(\mathsf{g})$ the {\em quantum group} associated to a symmetrizable Cartan datum which is generated by $e_i,f_i$ $(i \in I)$ and
$q^{h}$ $(h \in \wl^\vee)$.

\subsection{Foldable $r$-cluster points}
Let us consider the Dynkin diagrams $\Delta$ of finite type, labeled by an index set $I$, and their automorphisms $\vee$.
By the Dynkin diagram automorphisms $\vee$,
we can obtain the Dynkin diagrams of finite type $BCFG$ as orbits of $\vee$$\colon$
\begin{align}
B_n \ (n \ge 2) \quad  &\longleftrightarrow \quad
\big( A_{2n-1}: \xymatrix@R=0.5ex@C=4ex{ *{\circ}<3pt> \ar@{-}[r]_<{1 \ \ }  &*{\circ}<3pt>
\ar@{-}[r]_<{2 \ \ }  &   {}
\ar@{.}[r] & *{\circ}<3pt> \ar@{-}[r]_>{\,\,\,\ 2n-2} &*{\circ}<3pt>\ar@{-}[r]_>{\,\,\,\, 2n-1} &*{\circ}<3pt> }, \ i^\vee = 2n-i \big) \label{eq: B_n} \allowdisplaybreaks \\
C_n \ (n \ge 3) \quad  &\longleftrightarrow \quad
\left( D_{n+1}: \raisebox{1em}{\xymatrix@R=0.5ex@C=4ex{
& & &  *{\circ}<3pt>\ar@{-}[dl]^<{ \  n} \\
*{\circ}<3pt> \ar@{-}[r]_<{1 \ \ }  &*{\circ}<3pt>
\ar@{.}[r]_<{2 \ \ } & *{\circ}<3pt> \ar@{.}[l]^<{ \ \ n-1}  \\
& & &   *{\circ}<3pt>\ar@{-}[ul]^<{\quad \ \  n+1} \\
}}, \ i^\vee = \begin{cases} i & \text{ if } i \le n-1, \\ n+1 & \text{ if } i = n, \\ n & \text{ if } i = n+1. \end{cases} \right) \label{eq: C_n} \allowdisplaybreaks \\
F_4 \quad  &\longleftrightarrow \quad
\left( E_{6}: \raisebox{2em}{\xymatrix@R=3ex@C=4ex{
& & *{\circ}<3pt>\ar@{-}[d]_<{\quad \ \  6} \\
*{\circ}<3pt> \ar@{-}[r]_<{1 \ \ }  &
*{\circ}<3pt> \ar@{-}[r]_<{2 \ \ }  &
*{\circ}<3pt> \ar@{-}[r]_<{3 \ \ }  &
*{\circ}<3pt> \ar@{-}[r]_<{4 \ \ }  &
*{\circ}<3pt> \ar@{-}[l]^<{ \ \ 5 } }}, \begin{cases} 1^\vee=5, \ 5^\vee=1 \\ 2^\vee=4, \ 4^\vee=2, \\ 3^\vee=3, \ 6^\vee=6 \end{cases}  \right) \label{eq: F_4} \allowdisplaybreaks \\
G_2 \quad  &\longleftrightarrow \quad
\left( D_{4}: \raisebox{1em}{\xymatrix@R=0.5ex@C=4ex{
& &   *{\circ}<3pt>\ar@{-}[dl]^<{ \ 3} \\
*{\circ}<3pt> \ar@{-}[r]_<{1 \ \ }  &*{\circ}<3pt>
\ar@{-}[l]^<{2 \ \ }   \\
& &    *{\circ}<3pt>\ar@{-}[ul]^<{\quad \ \  4} \\
}}, \ \begin{cases} 1^\vee=3, \ 3^\vee=4, \ 4^\vee=1, \\ 2^\vee=2. \end{cases} \right) \label{eq: G_2}
\end{align}

Let $\W$ be the Weyl group, generated by simple reflections
$( s_i \ | \ i \in I)$ corresponding to $\Delta$,
and $w_0$ the longest element of $\W$. We denote by
$^*$, the involution on $I$ defined by
\eq
&&w_0(\al_i)=-\al_{i^*}.\label{eq:invI}
\eneq
 We also denote by $\PR$ the set of all positive roots.

\begin{definition}
We say that two reduced expressions
$\widetilde{w}=s_{i_1}s_{i_2}\cdots s_{i_{\ell}}$ and
$\widetilde{w}'=s_{j_1}s_{j_2}\cdots s_{j_{\ell}}$ of $w
\in \W$ are {\em commutation equivalent}, denoted by $\widetilde{w}
\sim \widetilde{w}' 
$, if
$s_{j_1}s_{j_2}\cdots s_{j_{\ell}}$ is obtained from
$s_{i_1}s_{i_2}\cdots s_{i_{\ell}}$ by applying the commutation
relations $s_{k}s_{l} =s_{l}s_{k}$ ($\lan h_k,\al_l\ran=0$).
We denote by $[\widetilde{w}]$ the commutation equivalence class
of $\widetilde{w}$.
\end{definition}

For each $[\redez]$, there exists a convex partial order $\prec_{[\redez]}$ on
$\Phi^+$, the set of positive roots,
satisfying the following property (see \cite{Oh15E} for details):
For $\al,\beta \in \PR$ with $\al+\beta \in \PR$, we have either
$$ \al \prec_{[\redez]}  \al+\beta \prec_{[\redez]} \beta \quad \text{ or } \quad \beta \prec_{[\redez]} \al+\beta \prec_{[\redez]} \al.$$

\begin{definition} Fix a Dynkin diagram $\Delta$ of finite type.
For an equivalence class $[\redez]$ of reduced expression $\redez$,
we say that  $i\in I$ is a {\em sink} (resp.\ {\em source}) of $[\redez]$
if there is a reduced expression $\redez'\in [\redez]$ of $w$ starting with $s_i$  (resp.\ ending with $s_i$).
\end{definition}

The following proposition is well-known (for example, see \cite{Kato12,OS15}):

\begin{proposition} For $\redez=s_{i_1}s_{i_2}\cdots s_{i_{\N-1}}s_{i_\N}$ with $i_\N^*=i$,
$\redez'=s_{i}s_{i_1}s_{i_2}\cdots s_{i_{\N-1}}$ is a reduced expression of $w_0$ and $[\redez'] \ne [\redez]$. Similarly,
$\redez''=s_{i_2}\cdots s_{i_{\N-1}}s_{i_\N}s_{i^*_1}$ is a reduced expression of $w_0$
and $[\redez''] \ne [\redez]$.
\end{proposition}

\begin{definition}[{\cite{OS15}}]
The right action of the reflection functor $r_i$ on $[\redez]$ is defined by
$$[\redez]\, r_i =
\bc
[(s_{i_2},\cdots, s_{i_N}, s_{i^*})] & \text{if there is $\redez'=(s_{i}, s_{i_2}, \cdots, s_{i_N})\in [\redez]$,}\\
\ [\redez] &  \text{otherwise.}
\ec
$$
On the other hand, the left right action of the reflection functor $r_i$ on $[\redez]$ is defined by
$$
r_i \, [\redez]=
\bc
[(s_{i^*},s_{i_1}\cdots, s_{i_{N-1}})] &  \text{if there is $\redez'=(s_{i_1}, \cdots, s_{i_{N-1}}, s_i)\in [\redez]$,}\\
\ [\redez] &  \text{otherwise.}
\ec
$$
\end{definition}

\begin{definition}[\cite{OS15}]  \label{def: ref equi}
 Let  $[\redez]$ and $[\redez']$ be two commutation classes. We say $[\redez]$ and $[\redez']$ are
{\em reflection equivalent} and write $[\redez]\overset{r}{\sim} [\redez']$ if $[\redez']$ can be obtained from $[\redez]$ by a
sequence of reflection maps. The equivalence class
$\lf \redez \rf \seteq \{\, [\redez]\, |\,
[\redez]\overset{r}{\sim}[\redez']\, \}$
with respect to the reflection equivalence relation  is called an {\em $r$-cluster point}.
\end{definition}

\begin{definition}[{\cite[Definition 1.8]{OS16B}}]
Fix a Dynkin diagram automorphism $\vee$ of $I$.
Let $I^\vee \seteq \{ \overline{i} \ | \ i \in I \}$ be the orbit classes of $I$ induced by $\vee$. For an $r$-cluster point
$\lf \redez \rf = \lf s_{i_1} \cdots s_{i_{\mathsf{N}}} \rf$ of $w_0$ and $\ok \in I^\vee$, define
$$\mathsf{C}^\vee_{\lf \redez \rf}(\ok) = |\{ i_s \ | \ i_s \in \ok, \  1 \le s \le \mathsf{N} \} |  \quad (\mathsf{N}=\ell(w_0)).$$
%
We call the composition
$\mathsf{C}^{\vee}_{\lf\redez\rf}=\big (\mathsf{C}^\vee_{\lf\redez\rf}(\overline{1}),\ldots, \mathsf{C}^{\vee}_{\lf\redez\rf} (\overline{|I^\vee|}) \big)$, 
the {\em $\vee$-Coxeter composition of $\lf \redez \rf$}.
%
\end{definition}

\begin{example}
For $\redez=s_1s_2s_3s_5s_4s_3s_1s_2s_3s_5s_4s_3s_1s_2s_3$ of type $A_5$, we have
$$\mathsf{C}^{\vee}_{\lf\redez\rf}=(5,5,5).$$
\end{example}

\begin{definition}  [{\cite[Definition 1.10]{OS16B}}] \label{def:foldable}
For an automorphism $\vee$ and an cluster
 $\lf \redez \rf$ of type $ADE$,
we say that an $r$-cluster point $\lf \redez \rf$ is {\em $\vee$-foldable} if
$$\mathsf{C}^{\vee}_{\lf \redez \rf}(\ok) = \mathsf{C}^{\vee}_{\lf \redez \rf}(\ol)  \qquad \text{ for any } \ok,\ol \in I^\vee.$$
\end{definition}

In \cite{OS16B,OS16C}, the existence for a $\vee$-foldable $r$-cluster point is proved, but we do not know whether it is unique or not:

\begin{proposition} [\cite{OS16B,OS16C}] \hfill
\begin{enumerate}
\item[{\rm (a)}] A $\vee$-foldable $r$-cluster point exists,
and is denoted by $\lf \mQ \rf$.
\item[{\rm (b)}] The number of commutation classes in each $\lf \mQ \rf$ is equal to $2^{|I|-|\vee|} \times |\vee|$,
where $|\vee|$ denotes the order of\/ $\vee$.
\end{enumerate}
\end{proposition}

The $r$-cluster point $\lf \mQ \rf$ is called a {\em twisted adapted cluster point} and a class $[\mQ]$ in $\lf \mQ \rf$ is called a {\em twisted adapted class}.

Let $\sigma \in GL(\C\Phi)$ be a linear transformation of finite order
which preserves $\Pi$. Hence $\sigma$ preserves
$\Phi$ itself and normalizes $\W$ and so $\W$ acts by conjugation on the
coset $\W\sigma$.

\begin{definition} \label{def: twisted Coxeter} \hfill
\begin{enumerate}
\item Let $\{ \Pi_{i_1}, \ldots,\Pi_{i_k} \}$ be the all orbits of $\Pi$ in
$\Phi$ with respect to $\sigma$. For each $r\in \{1, \cdots, k \}$, choose $\alpha_{i_r} \in \Pi_{i_r}$ arbitrarily,
and let $s_{i_r} \in \W$ denote the corresponding reflection. Let $w$
be the product of $s_{i_1}, \ldots , s_{i_k}$ in any order. The
element $w\sigma \in \W\sigma$ of $w\in \W$ thus obtained is called a {\it
$\sigma$-Coxeter element}.
\item If $\sigma$ in (1) is $\vee$ in \eqref{eq: B_n}, \eqref{eq: C_n}, \eqref{eq: F_4},
then  {\it $\sigma$-Coxeter element} is also called a {\it twisted Coxeter element}.
\end{enumerate}
\end{definition}

\begin{remark} \hfill
\begin{enumerate}
\item[{\rm (i)}] For the involution of $A_{2n-1}$ in \eqref{eq: B_n}, of $D_{n+1}$ in
\eqref{eq: C_n} and of $E_6$ in \eqref{eq: F_4}, the number of
commutation classes  of each $\lf Q \rf$ and the one of each $\lf \mQ \rf$ are
the same and are equal to $2^{2n-2}$, $2^n$ and $2^5$, respectively.
\item[{\rm (ii)}] For types $A_{2n-1}$, $D_{n+1}$ and $E_6$, $\lf \mQ \rf$ are given as follows:
$$
\lf \mQ \rf =\begin{cases}
\displaystyle \left[\hspace{-0.7ex}\left[\prod_{k=0}^{2n-2} (s_{j_1}s_{j_2}s_{j_3}\cdots s_{j_n})^{k\vee}\right]\hspace{-0.7ex}\right]  & \text{ if $\Delta$ is of type $A_{2n-1}$}, \\[3ex]
\displaystyle \left[\hspace{-0.7ex}\left[\prod_{k=0}^{n} (s_{j_1}s_{j_2}s_{j_3}\cdots s_{j_n})^{k\vee}\right]\hspace{-0.7ex}\right] & \text{ if $\Delta$ is of type $D_{n+1}$},\\[3ex]
\displaystyle \left[\hspace{-0.7ex}\left[\prod_{k=0}^{8} (s_{j_1}s_{j_2}s_{j_3}s_{j_4})^{k\vee}\right]\hspace{-0.7ex}\right] & \text{ if $\Delta$ is of type $E_{6}$},
\end{cases}
$$
where
\begin{itemize}
\item $s_{j_1}s_{j_2}s_{j_3}\cdots s_{j_n}$ is an arbitrary twisted Coxeter element of type $A_{2n-1}$ (resp. $D_{n+1}$ and $E_6$),
\item $\vee$ is given in \eqref{eq: B_n} (resp. \eqref{eq: C_n} and \eqref{eq: F_4}),
\item $(s_{j_1} \cdots s_{j_n})^\vee \seteq s_{j^\vee_1} \cdots s_{j^\vee_n}$ and $(s_{j_1} \cdots s_{j_n})^{k \vee} \seteq(\cdots((s_{j_1} \cdots s_{j_n} \underbrace{ )^\vee )^\vee \cdots
 )^\vee}_{\text{ $k$-times}}$.
\end{itemize}
Note that $s_1s_2s_3\cdots s_{n}$ is a twisted Coxeter element of type $A_{2n-1}$, $D_{n+1}$ and $E_6$.
\item [{\rm (iii)}] For types $D_{4}$, $\lf \mQ \rf$ with respect to \eqref{eq: G_2} is given as follows:
$$\lf \mQ \rf = \displaystyle \left[\hspace{-0.7ex}\left[\prod_{k=0}^{5} (s_{2}s_{1})^{k\vee}\right]\hspace{-0.7ex}\right]$$
\end{enumerate}
\end{remark}

\begin{example} \hfill
\begin{enumerate}
\item[{\rm (i)}] For $\vee$ in \eqref{eq: B_n}, the Coxeter composition of a foldable cluster is
$$ \mathsf{C}^{\vee}_{\lf \mQ  \rf}=( \underbrace{2n-1, \ldots ,2n-1}_{ n\text{-times}} ).$$
\item[{\rm (ii)}] For $\vee$ in \eqref{eq: C_n}, the Coxeter composition of a foldable cluster is
$$ \mathsf{C}^{\vee}_{\lf \mQ  \rf}=( \underbrace{n+1, \ldots ,n+1}_{ n\text{-times}} ).$$
\item[{\rm (iii)}] For $\vee$ in \eqref{eq: F_4}, the Coxeter composition of a foldable cluster is
$\mathsf{C}^{\vee}_{\lf \mQ  \rf}=(9,9,9,9)$.
\item[{\rm (iv)}] For $\vee$ in \eqref{eq: G_2}, the Coxeter composition of a foldable cluster is
$\mathsf{C}^{\vee}_{\lf \mQ  \rf}=(6,6)$.
\end{enumerate}
\end{example}

\subsection{Adapted cluster point and Auslander-Reiten quiver}
Let $Q$ be a Dynkin quiver by orienting edges of a Dynkin diagram $\Delta$ of type $ADE$.
We say that a vertex $i$ in $Q$ is a {\em source} (resp.\ {\em sink}) if and only if
there are only exiting arrows out of it (resp.\ entering arrows into it).
For a source (resp.\ sink) $i$,
If $i$ is a sink or source, $\re_iQ$ denotes the quiver obtained by
$Q$ by reversing the arrows incident with $i$.
We say that a reduced expression $\tw=\re_{i_1}\re_{i_2}\cdots \re_{i_{\ell(w)}}$
 of $w \in \W$  is
{\em adapted to} $Q$ if $i_k$ is a sink of the quiver $\re_{i_{k-1}} \cdots \re_{i_2}\re_{i_1}Q$ for all $1 \le k \le \ell(w)$.

\medskip

The followings are well-known:

\begin{theorem} \label{thm: Qs} \hfill
\begin{enumerate}
\item[{\rm (1)}] Any reduced word $\redez$ of $w_0$ is adapted
to at most one Dynkin quiver $Q$.
\item[{\rm (2)}] For each Dynkin quiver $Q$, there is a reduced word $\redez$ of $w_0$ adapted to $Q$. Moreover,
any reduced word $\redez'$ in $[\redez]$ is adapted to $Q$,
and the commutation equivalence class $[\redez]$
is uniquely determined by $Q$.
We denote by $[Q]$ of the commutation equivalence class $[\redez]$.
\item[{\rm (3)}] For every commutation class $[Q]$, there exists a unique Coxeter element $\phi_Q$
which is a product of all simple reflections and adapted to $Q$.
\item[{\rm (4)}] For every Coxeter element $\phi$, there exists a unique Dynkin quiver $Q$ such that $\phi=\phi_Q$.
\item[{\rm (5)}] All commutation classes $\{ [Q] \}$ are reflection equivalent and form the $r$-cluster point $\lf Q \rf$, called the {\it adapted cluster point}. The number of
commutation classes in $\lf Q \rf$ is $2^{|I|-1}$. \\
\end{enumerate}
\end{theorem}

Let $\Phi(\phi_Q)$ be the subset of $\PR$ determined by $\phi_Q=s_{i_1}s_{i_2}\cdots s_{i_n}$ with $|I|=n$:
$$  \Phi(\phi_Q)=\PR\cap \phi_Q\PR=
\{ \beta^{\phi_Q}_1=\al_{i_1}, \beta^{\phi_Q}_2=s_{i_1}(\al_{i_1}),\ldots,\beta^{\phi_Q}_n =s_{i_1}\cdots s_{i_{n-1}}(\al_{i_n}) \}.$$

The {\it height function} $\xi$ on $Q$ is an integer-valued map $\xi:Q \to \Z$ satisfying $\xi(j)=\xi(i)+1$ when $i \to j$ in $Q$.

\medskip

The Auslander-Reiten quiver (AR-quiver) $\Gamma_Q$ associated to $Q$ is a quiver with coordinates in $I \times \Z$ defined as follows: Construct an injective map
$\Omega_Q: \PR\to I \times \Z$ in an inductive way.
\begin{itemize}
\item[{\rm (i)}] $\Omega_Q(\beta_k^{\phi_Q}) \seteq (i_k,\xi(i_k))$.
\item[{\rm (ii)}] If $\Omega_Q(\beta)$ is already assigned as $(i,p)$ and $\phi_Q(\beta) \in \PR$, then $\Omega_Q(\phi_Q(\beta))=(i,p-2)$.
\end{itemize}
The AR-quiver $\Gamma_Q$ is a quiver whose vertices consist of ${\rm Im}(\Omega_Q)$ ($\simeq\PR$) and arrows
 $(i,p) \to (j,q)$ are assigned when $i$ and $j$ are adjacent in $\Delta$ and $p-q=-1$.

\begin{example} \label{example: A_4}
The AR-quiver $\Gamma_Q$ associated to $\xymatrix@R=3ex{ *{ \bullet }<3pt> \ar@{<-}[r]_<{1}  &*{\bullet}<3pt>
\ar@{->}[r]_<{2}  &*{\bullet}<3pt>
\ar@{->}[r]_<{3} &*{\bullet}<3pt>
\ar@{-}[l]^<{4} }$ of type $A_4$ with the height function such that $\xi(1)=0$ is given as follows:
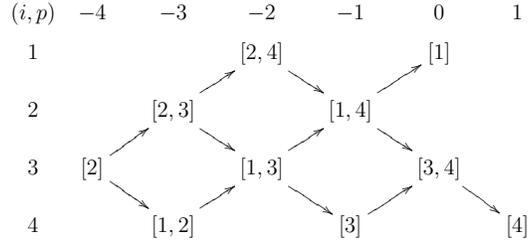
\[ \scalebox{0.7}{\ \xymatrix@C=1ex@R=0.5ex{
( i,p ) &-4&&-3&&-2&&-1&&0 && 1 \\
1&&&&& [2,4] \ar@{->}[ddrr] &&&&[1]  \\  & \\
2&&&[2,3]\ar@{->}[uurr]\ar@{->}[ddrr] && && [1,4] \ar@{->}[uurr]\ar@{->}[ddrr]\\ & \\
3& [2]\ar@{->}[uurr]\ar@{->}[ddrr] &&&& [1,3]\ar@{->}[uurr]\ar@{->}[ddrr] &&&& [3,4]\ar@{->}[ddrr]\\ & \\
4&&& [1,2] \ar@{->}[uurr] &&&& [3] \ar@{->}[uurr] &&&& [4]}}
\]
Here $[a,b]$ $(1 \le a,b \le 4)$ stands for the positive root $\sum_{k=a}^b \al_k$ of $\PR_{A_4}$.
\end{example}

Interestingly, $\Gamma_Q$ can be understood as a visualization of $\prec_{Q} \seteq \prec_{[Q]}$ and is closely related to the commutation class $[Q]$:

\begin{theorem} \cite{B99,OS15}
\begin{enumerate}
\item[{\rm (1)}] $\al \prec_Q \beta$ if and only if there exists a path from $\beta$ to $\al$ inside of $\Gamma_Q$.
\item[{\rm (2)}] By reading the residues
\ro i.e., $i$ for $(i,p)$\rof\  of vertices in a way {\em compatible with} arrows, we can obtain all reduced expressions $\redez \in [Q]$.
\end{enumerate}
\end{theorem}

In Example \ref{example: A_4}, we can get a reduced expression $\redez$ in $[Q]$ as follows:
$$\redez=
s_4s_1s_3s_2s_4s_1s_3s_2s_4s_3$$

\subsection{Relationship between $\lf \mQ \rf$ and $\lf Q \rf$} In this subsection, we briefly recall the relationship between
$\lf \mQ \rf$ and $\lf Q \rf$ studied in \cite{OS16B,OS16C}.
We shall first consider a Dynkin quiver $Q$ of type $A$ and $\redez = s_{i_1}s_{i_2} \cdots s_{i_\N}$ in $[Q]$.

\begin{theorem} [\cite{OS16B}]
For $\redez=s_{i_1}s_{i_2} \cdots s_{i_\N} \in [Q]$ of type
$A_{2n-2}$, we
can obtain two distinct twisted adapted classes $[\mQ^>], [\mQ^<]
\in \lf \mQ \rf$ of type $A_{2n-1}$ as follows:
\begin{enumerate}
\item[{\rm (1)}] For each pair $(i_k,i_l)$ such that $\{ i_k ,i_l \} = \{ n-1, n \}$ and $i_j \not \in \{ n-1, n \}$ for any $j$ with $k < j <l$,
we replace subexpression
$ s_{i_k}s_{i_{k+1}} \cdots  s_{i_l}$ with $ s_{i^+_k}s_{n}s_{i^+_{k+1}} \cdots  s_{i^+_l}$ where $i^+=i+1$ if $i > n-1$ and $i^+=i$ otherwise.
\item[{\rm (2)}] For the smallest index $($resp.\ the largest index$)$ $i_t$  with $i_t  \in \{ n-1,n \}$, we replace $s_{i_t}$ with $s_{n}s_{i^+_t}$
$($resp.\ $s_{i^+_t}s_{n})$.
\end{enumerate}
Then the resulted reduced expression $\redez^>$ $($resp.\ $\redez^<)$ is a reduced expression whose commutation class $[\mQ^>]$ $($resp.\ $[\mQ^<])$
is well-defined and twisted adapted. Conversely, each commutation class in $\lf \mQ\rf$ can be obtained in this way and $[\mQ] \ne [\mathscr{Q'}]$
if $Q \ne Q'$.
\end{theorem}

By the work of \cite{OS15}, the combinatorial AR-quivers $\Upsilon_{[\mQ^>]}$ and $\Upsilon_{[\mQ^<]}$ of $Q$ in Example \ref{example: A_4} can be understood as
realization of the convex partial orders $\prec_{[\mQ^>]}$ and $\prec_{[\mQ^<]}$:
\begin{align}\label{eq: Ex A}
\scalebox{0.64}{\xymatrix@C=1ex@R=1ex{
( i,p ) &-4&-\frac{7}{2}&-3&-\frac{5}{2}&-2&-\frac{3}{2}&-1&-\frac{1}{2}&0 &\frac{1}{2}& 1 \\
1&&&&& \bullet \ar@{->}[drr] &&&& \bullet  \\
2&&&\bullet\ar@{->}[urr]\ar@{->}[dr] && && \bullet \ar@{->}[urr]\ar@{->}[dr]\\
3 && \bigstar\ar@{->}[ur] && \bigstar\ar@{->}[dr] && \bigstar\ar@{->}[ur] && \bigstar\ar@{->}[dr] && \bigstar \\
4& \bullet\ar@{->}[ur]\ar@{->}[drr] &&&& \bullet\ar@{->}[ur]\ar@{->}[drr] &&&& \bullet\ar@{->}[drr]\ar@{->}[ur]\\
5&&& \bullet \ar@{->}[urr] &&&& \bullet \ar@{->}[urr] &&&& \bullet}} \
\scalebox{0.64}{\xymatrix@C=1ex@R=1ex{
( i,p ) &-\frac{9}{2}&-4&-\frac{7}{2}&-3&-\frac{5}{2}&-2&-\frac{3}{2}&-1&-\frac{1}{2}&0 &\frac{1}{2}& 1 \\
1&&&&&& \bullet \ar@{->}[drr] &&&& \bullet  \\
2&&&&\bullet\ar@{->}[urr]\ar@{->}[dr] && && \bullet \ar@{->}[urr]\ar@{->}[dr]\\
3&\bigstar\ar@{->}[dr]&& \bigstar\ar@{->}[ur] && \bigstar\ar@{->}[dr] && \bigstar\ar@{->}[ur] && \bigstar\ar@{->}[dr]  \\
4&& \bullet\ar@{->}[ur]\ar@{->}[drr] &&&& \bullet\ar@{->}[ur]\ar@{->}[drr] &&&& \bullet\ar@{->}[drr]\\
5&&&& \bullet \ar@{->}[urr] &&&& \bullet \ar@{->}[urr] &&&& \bullet}}
\end{align}
For each new vertex, denoted by $\bigstar$ above, we can assign its coordinate in $I \times \Z/2$ in a canonical way. By \cite{OS15},
we can obtain all reduced expressions $\redez \in [\mQ^>]$ (resp.\ $[\mQ^<]$) by reading it in a compatible way with arrows: For instances, we have
\begin{itemize}
\item $s_5s_3s_1s_4s_3s_5s_3s_1s_4s_3s_2s_5s_3s_4\in [\mQ^>]$.
\item $s_5s_4s_1s_3s_2s_3s_5s_4s_1s_3s_2s_3s_5s_4s_3 \in [\mQ^<]$.
\end{itemize}

In $D_{n+1}$ case, we can get two distinct commutation classes $[\mQ^{\gets n}]$ and $[\mQ^{\gets n+1}] \in \lf \mQ \rf$
from $\Gamma_Q$ of type $A_n$:

\begin{theorem}[\cite{OS16C}]\hfill
\begin{itemize}
\item[{\rm (1)}] For a given $\Gamma_Q$, consider the copy $\Gamma^{\updownarrow}_Q$ of $\Gamma_Q$ by turning upside down.
\item[{\rm (2)}] By putting $\Gamma^{\updownarrow}_Q$ to the left of $\Gamma_Q$, we have new quiver inside of $I \times \Z$ by assigning arrows to vertices $(i,p)$ and
$(j,q) \in \Gamma^{\updownarrow}_Q \sqcup \Gamma_Q$, $(i,p)\to (j,q)$ such that $i,j$ are adjacent in $Q$ and
$q-p=1$.
\item[{\rm (3)}] For vertices whose residues are $n$, we change their residues as $n,n+1,n,n+1...$
$($resp.\ $n+1,n,n+1,n...)$ from the right-most one.
\end{itemize}
Then the resulted quiver coincides with the combinatorial quiver $\Upsilon_{[\mQ^{\gets n}]}$
$($resp.\ $\Upsilon_{[\mQ^{\gets n+1}]})$ of type $D_{n+1}$ introduced in \cite{OS15,OS16C}. Thus we can obtain all reduced expressions $\redez \in [\mQ^{\gets n}]$
$($resp.\ $[\mQ^{\gets n+1}])$ by reading it. Conversely, each commutation class in $\lf \mQ \rf$ can be obtained in this way and
$[\mQ] \ne [\mathscr{Q'}]$
if $Q \ne Q'$.
\end{theorem}

\begin{example} \label{Ex: UpQ}
For better explanation of the above theorem, we now give examples by using $\Gamma_Q$ in Example \ref{example: A_4}: For the $\Gamma_Q$, $\Gamma^{\updownarrow}_Q$
can be described as follows:
$$
\Gamma^{\updownarrow}_Q =
\raisebox{2em}{\scalebox{0.8}{\xymatrix@C=2ex@R=1ex{
1 &&  \bullet \ar@{->}[dr] && \bullet\ar@{->}[dr] && \bullet \\
2 & \bullet\ar@{->}[dr]\ar@{->}[ur] && \bullet\ar@{->}[dr]\ar@{->}[ur] && \bullet\ar@{->}[ur]  \\
3 && \bullet \ar@{->}[dr]\ar@{->}[ur]  && \bullet  \ar@{->}[dr]\ar@{->}[ur]  \\
4 &&& \bullet \ar@{->}[ur] && \bullet}}}
\qquad
\Gamma_Q =
\raisebox{2em}{\scalebox{0.8}{\xymatrix@C=2ex@R=1ex{
1 &&&  \bullet \ar@{->}[dr] && \bullet \\
2 && \bullet\ar@{->}[dr]\ar@{->}[ur] && \bullet\ar@{->}[dr]\ar@{->}[ur] \\
3 & \bullet \ar@{->}[dr]\ar@{->}[ur]  && \bullet  \ar@{->}[dr]\ar@{->}[ur]  && \bullet  \ar@{->}[dr] \\
4 && \bullet \ar@{->}[ur] && \bullet \ar@{->}[ur]&& \bullet }}}
$$
By putting
$\Gamma^{\updownarrow}_Q$ to the left of $\Gamma_Q$, we have new quiver as follows:
$$\scalebox{0.8}{\xymatrix@C=2ex@R=1ex{
( i,p ) & -9 & -8 & -7 & -6 & -5 & -4 & -3& -2 &-1 & 0 & 1 \\
1 &&  \bullet \ar@{->}[dr] && \bullet\ar@{->}[dr] && \bullet \ar@{->}[dr] &&  \bullet \ar@{->}[dr] &&  \bullet \\
2 & \bullet\ar@{->}[dr]\ar@{->}[ur] && \bullet\ar@{->}[dr]\ar@{->}[ur] && \bullet\ar@{->}[ur]\ar@{->}[dr] &&  \bullet \ar@{->}[ur]\ar@{->}[dr]&&  \bullet \ar@{->}[dr]\ar@{->}[ur] \\
3 && \bullet \ar@{->}[dr]\ar@{->}[ur]  && \bullet  \ar@{->}[dr]\ar@{->}[ur] &&  \bullet \ar@{->}[ur]\ar@{->}[dr]&&  \bullet \ar@{->}[ur]\ar@{->}[dr] &&  \bullet \ar@{->}[dr]\\
4 &&& \bullet \ar@{->}[ur] && \bullet \ar@{->}[ur]&&  \bullet \ar@{->}[ur]&&  \bullet \ar@{->}[ur]&&  \bullet}}
$$
Now we can get $\Upsilon_{[\mQ^{\gets n}]}$ and $\Upsilon_{[\mQ^{\gets n+1}]}$ as follows:
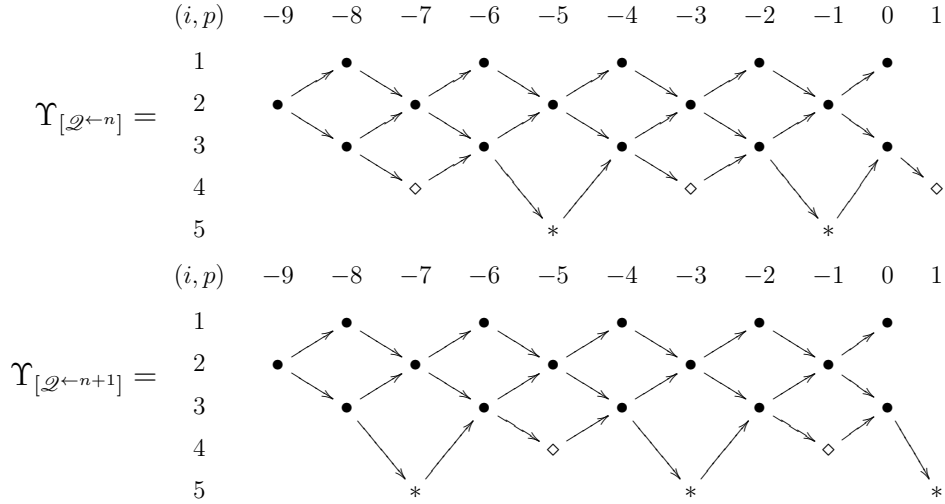
\begin{align*}
  \Upsilon_{[\mQ^{\gets n}]} & =\raisebox{3.5em}{\scalebox{0.8}{\xymatrix@C=2ex@R=1ex{
( i,p ) & -9& -8 & -7 & -6 & -5 & -4 & -3& -2 &-1 & 0 & 1 \\
1 &&  \bullet \ar@{->}[dr] && \bullet\ar@{->}[dr] && \bullet \ar@{->}[dr] &&  \bullet \ar@{->}[dr] &&  \bullet \\
2 & \bullet\ar@{->}[dr]\ar@{->}[ur] && \bullet\ar@{->}[dr]\ar@{->}[ur] && \bullet\ar@{->}[ur]\ar@{->}[dr] &&  \bullet \ar@{->}[ur]\ar@{->}[dr]&&  \bullet \ar@{->}[dr]\ar@{->}[ur] \\
3 && \bullet \ar@{->}[dr]\ar@{->}[ur]  && \bullet  \ar@{->}[ddr]\ar@{->}[ur] &&  \bullet \ar@{->}[ur]\ar@{->}[dr]&&  \bullet \ar@{->}[ur]\ar@{->}[ddr] &&  \bullet \ar@{->}[dr]\\
4 &&& \diamond \ar@{->}[ur] &&&&  \diamond \ar@{->}[ur]&&&&  \diamond  \\
5 &&&&& \ast \ar@{->}[uur]&&&& \ast \ar@{->}[uur]}}}
\\
\Upsilon_{[\mQ^{\gets n+1}]} & = \raisebox{3.5em}{\scalebox{0.8}{\xymatrix@C=2ex@R=1ex{
( i,p ) & -9 & -8 & -7 & -6 & -5 & -4 & -3& -2 &-1 & 0 & 1 \\
1 &&  \bullet \ar@{->}[dr] && \bullet\ar@{->}[dr] && \bullet \ar@{->}[dr] &&  \bullet \ar@{->}[dr] &&  \bullet \\
2 & \bullet\ar@{->}[dr]\ar@{->}[ur] && \bullet\ar@{->}[dr]\ar@{->}[ur] && \bullet\ar@{->}[ur]\ar@{->}[dr] &&  \bullet \ar@{->}[ur]\ar@{->}[dr]&&  \bullet \ar@{->}[dr]\ar@{->}[ur] \\
3 && \bullet \ar@{->}[ddr]\ar@{->}[ur]  && \bullet  \ar@{->}[dr]\ar@{->}[ur] &&  \bullet \ar@{->}[ur]\ar@{->}[ddr]&&  \bullet \ar@{->}[ur]\ar@{->}[dr] &&  \bullet \ar@{->}[ddr]\\
4 &&&  && \diamond \ar@{->}[ur]&&&&  \diamond \ar@{->}[ur]   \\
5 &&& \ast \ar@{->}[uur] && &&  \ast \ar@{->}[uur]&& &&  \ast }}}
\end{align*}
One can easily notice that we can assign a coordinate to each vertex in a canonical way.
\end{example}

\subsection{Folded AR-quivers} Now we can define {\em a folded $AR$-quiver} $\WUp_{[\mQ]}$ associated to the commutation class $[\mQ]$
in $\lf \mQ \rf$ of type $A_{2n-1}$ and $D_{n+1}$ by folding $\Upsilon_{[\mQ]}$ (\cite{OS16B,OS16C}):
\begin{enumerate}
\item[{\rm (i)}] ($[\mQ]$ of type $A_{2n-1}$) By replacing coordinate $\Omega_Q(\beta)=(i,p/2)$ of $\beta$ in $\Upsilon_{[\mQ]}$ with
$\widehat{\Omega}_{[\mQ]}(\beta)=(\overline{i},p) \in \overline{I} \times \Z$, we have new quiver $\WUp_{[\mQ]}$ with {\em folded coordinates}, which is isomorphic to
$\Upsilon_{[\mQ]}$ as quivers.
\item[{\rm (ii)}] ($[\mQ]$ of type $D_{n+1}$) By replacing coordinate $\Omega(\beta)=(i,p)$ of $\beta$ in $\Upsilon_{[\mQ]}$ with
$\widehat{\Omega}_{[\mQ]}(\beta)=(\overline{i},p) \in \overline{I} \times \Z$, we have new quiver $\WUp_{[\mQ]}$ with folded coordinates.
\end{enumerate}
We call $\WUp_{[\mQ]}$ the {\em folded $AR$-quiver} associated to $[\mQ]$ and $\widehat{\Omega}_{[\mQ]}(\beta)=(\overline{i},p)$
the {\em folded coordinate} of $\beta$ with respect to $[\mQ]$.

\begin{example}
From \eqref{eq: Ex A} and Example \ref{Ex: UpQ}, we can obtain folded AR-quivers $\WUp_{[\mQ]}$ as follows:
\begin{enumerate}
\item[{\rm (1)}] $[\mQ]$ of type $A_{5}$ cases:
$$
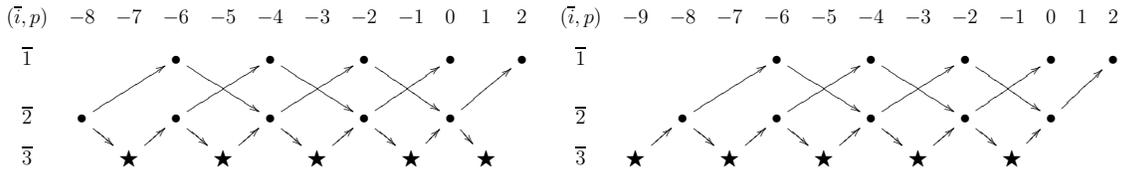
\scalebox{0.65}{\xymatrix@C=1ex@R=1ex{
( \overline{i},p ) &-8&-7&-6&-5&-4&-3&-2&-1&0 &1& 2 \\
\overline{1}&&&\bullet\ar@{->}[ddrr]&& \bullet \ar@{->}[ddrr] &&\bullet\ar@{->}[ddrr]&& \bullet  &&\bullet\\  & \\
\overline{2}&\bullet\ar@{->}[dr]\ar@{->}[uurr]&&\bullet\ar@{->}[uurr]\ar@{->}[dr] &&\bullet\ar@{->}[dr]\ar@{->}[uurr]&& \bullet \ar@{->}[uurr]\ar@{->}[dr] && \bullet\ar@{->}[uurr]\ar@{->}[dr]\\
\overline{3} && \bigstar\ar@{->}[ur] && \bigstar\ar@{->}[ur] && \bigstar\ar@{->}[ur] && \bigstar\ar@{->}[ur] && \bigstar }} \ \
\scalebox{0.65}{\xymatrix@C=1ex@R=1ex{
( \overline{i},p ) &-9&-8&-7&-6&-5&-4&-3&-2&-1&0 &1& 2 \\
\overline{1}&&&&\bullet \ar@{->}[ddrr]&& \bullet \ar@{->}[ddrr] &&\bullet \ar@{->}[ddrr]&& \bullet && \bullet &&&& \\  & \\
\overline{2}&&\bullet\ar@{->}[dr]\ar@{->}[uurr]&&\bullet\ar@{->}[uurr]\ar@{->}[dr] &&\bullet\ar@{->}[dr]\ar@{->}[uurr]&& \bullet \ar@{->}[uurr]\ar@{->}[dr]&& \bullet\ar@{->}[uurr]\\
\overline{3}&\bigstar\ar@{->}[ur]&& \bigstar\ar@{->}[ur] && \bigstar\ar@{->}[ur] && \bigstar\ar@{->}[ur] && \bigstar\ar@{->}[ur]}}
$$
\item[{\rm (2)}] $[\mQ]$ of type $D_{5}$ cases:
$$
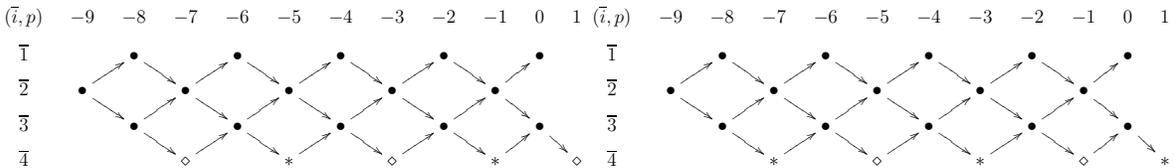
\scalebox{0.6}{\xymatrix@C=2ex@R=1ex{
( \overline{i},p ) & -9 & -8 & -7 & -6 & -5 & -4 & -3& -2 &-1 & 0 & 1 \\
\overline{1} &&  \bullet \ar@{->}[dr] && \bullet\ar@{->}[dr] && \bullet \ar@{->}[dr] &&  \bullet \ar@{->}[dr] &&  \bullet \\
\overline{2} & \bullet\ar@{->}[dr]\ar@{->}[ur] && \bullet\ar@{->}[dr]\ar@{->}[ur] && \bullet\ar@{->}[ur]\ar@{->}[dr] &&  \bullet \ar@{->}[ur]\ar@{->}[dr]&&  \bullet \ar@{->}[dr]\ar@{->}[ur] \\
\overline{3} && \bullet \ar@{->}[dr]\ar@{->}[ur]  && \bullet  \ar@{->}[dr]\ar@{->}[ur] &&  \bullet \ar@{->}[ur]\ar@{->}[dr]&&  \bullet \ar@{->}[ur]\ar@{->}[dr] &&  \bullet \ar@{->}[dr]\\
\overline{4} &&& \diamond \ar@{->}[ur] && \ast \ar@{->}[ur]&&  \diamond \ar@{->}[ur]&&  \ast \ar@{->}[ur]&&  \diamond}}
\scalebox{0.6}{\xymatrix@C=2ex@R=1ex{
( \overline{i},p ) & -9 & -8 & -7 & -6 & -5 & -4 & -3& -2 &-1 & 0 & 1 \\
\overline{1} &&  \bullet \ar@{->}[dr] && \bullet\ar@{->}[dr] && \bullet \ar@{->}[dr] &&  \bullet \ar@{->}[dr] &&  \bullet \\
\overline{2} & \bullet\ar@{->}[dr]\ar@{->}[ur] && \bullet\ar@{->}[dr]\ar@{->}[ur] && \bullet\ar@{->}[ur]\ar@{->}[dr] &&  \bullet \ar@{->}[ur]\ar@{->}[dr]&&  \bullet \ar@{->}[dr]\ar@{->}[ur] \\
\overline{3} && \bullet \ar@{->}[dr]\ar@{->}[ur]  && \bullet  \ar@{->}[dr]\ar@{->}[ur] &&  \bullet \ar@{->}[ur]\ar@{->}[dr]&&  \bullet \ar@{->}[ur]\ar@{->}[dr] &&  \bullet \ar@{->}[dr]\\
\overline{4} &&& \ast \ar@{->}[ur] && \diamond \ar@{->}[ur]&&  \ast \ar@{->}[ur]&&  \diamond \ar@{->}[ur]&&  \ast}}
$$
\end{enumerate}
\end{example}

\section{Positive root systems} In this section, we recall main results of \cite{Oh15E,OS16B,OS16C}, which investigated the positive root systems
by using newly introduced notions and (combinatorial) AR-quivers.

\subsection{Notions} For a reduced expression $\redez=s_{i_1}s_{i_2}\cdots s_{i_\N}$ of $w_0$, there exists the convex total order $<_{\redez}$ on $\PR$
defined as follows:
$$ \beta^{\redez}_{k} <_{\redez} \beta^{\redez}_{l}
\quad \text{if and only if $k <l$,}$$
where $\beta^{\redez}_{k} \seteq s_{i_1}\cdots s_{i_{k-1}}(\al_{i_k})$.

With the convex total order $<_{\redez}$, we identify
$\um=(\um_1,\um_2,\ldots,\um_\N) \in \Z_{\ge 0}^{\N}$  with
$$ \um_{\redez}\in\ (\Z_{\ge 0})^{|\PR|},$$ whose coordinate at $\beta^{\redez}_{k}$
is $m_k$.
For a sequence $\um$, we set $\wt(\um)=\sum_{i=1}^\N \um_i\beta^{\redez}_i\in\rl^+$.


\begin{definition}[\cite{Mc12,Oh15E}] We define the partial orders $<^\tb_{\redez}$ and $\prec^\tb_{[\redez]}$ on $\Z_{\ge 0}^{\N}$ as follows:
\begin{enumerate}
\item[{\rm (i)}] $<^\tb_{\redez}$ is the bi-lexicographical partial order induced by $<_{\redez}$. Namely, $\um<^\tb_{\redez}\um'$ if there exist
$j$ and $k$ ($1\le j\le k\le N$)
such that $\um_s=\um'_s$ for $1\le s<j$, $\um_j<\um'_j$
and $\um_{s}=\um'_{s}$ for $k<s\le N$, $\um_k<\um'_k$.
\item[{\rm (ii)}] For sequences $\um$ and $\um'$, $\um \prec^\tb_{[\redez]} \um'$ if and only if $\wt(\um)=\wt(\um')$ and
$\un<^\tb_{\redez'} \un'$ for all $\redez' \in [\redez]$,
where $\un$ and $\un'$ are sequences such that
$\un_{\redez'}=\um_{\redez}$ and $\un'_{\redez'}=\um_{\redez}$.
\end{enumerate}
\end{definition}


We call a sequence $\um$ a {\em pair} if $|\um|\seteq \sum_{i=1}^\N m_i=2$ and $m_i \le 1$ for $1\le i\le \N$. We mainly use the notation $\up$ for a pair.
We also write $\up$ as  $(\beta^{\redez}_{i_1},\beta^{\redez}_{i_2})$ or $(i_1,i_2)$
where $\up_{i_1}=\up_{i_2}=1$ and $i_1 \le i_2$.

\begin{definition} [\cite{Oh15E}]\hfill
\begin{enumerate}
\item[{\rm (i)}] A pair $\up$ is called {\it $[\redez]$-simple} if there exists no sequence $\um \in \Z^{\N}_{\ge 0}$ satisfying
$\um \prec^\tb_{[\redez]} \up$.
\item[{\rm (ii)}] A sequence $\um=(\um_1,\um_2,\ldots,\um_{\N}) \in \Z^{\N}_{\ge 0}$ is called {\it $[\redez]$-simple} if $\um=(\um_k \beta^\redez_k)$ for some
$1\le k \le \N$ or any pair $({i_1},{i_2})$
such that $\um_{i_1},\um_{i_2} >0$ is a $[\redez]$-simple pair.
\end{enumerate}
\end{definition}

\begin{definition}[\cite{Mc12,Oh15E}]  For a given $[\redez]$-simple sequence $\us=(s_1,\ldots,s_{\N}) \in \Z^{\N}_{\ge 0}$,
we say that a sequence $\um \in \Z^{\N}_{\ge 0}$ is called a {\em $[\redez]$-minimal sequence of $\us$} if $\um$ satisfies the following properties:
$$\text{$\us \prec^\tb_{\redez} \um$ and there exists no sequence $\um'\in \Z^{\N}_{\ge 0}$ such that
$\us \prec^\tb_{[\redez]} \um' \prec^\tb_{[\redez]}  \um$.}$$
\end{definition}

\begin{definition}[\cite{Oh15E}] \label{def: gdist}
The {\em $[\redez]$-distance} of a sequence $\um$,
denoted by $\dist_{[\redez]}(\um)$,
is the largest integer $k\ge0$ such that
there exists a family of
sequences $\{\um^{(i)}\}_{0\le i\le k}$ satisfying
$$\um^{(0)} \prec^\tb_{[\redez]} \cdots \prec^\tb_{[\redez]} \um^{(k)}=\um.$$
Note that $\um^{(0)} $ should be $[\redez]$-simple.
\end{definition}

\begin{definition} [\cite{Oh15E}] \label{def: jj_0-socle}
For a pair $\up$, the {\em $[\redez]$-socle} of $\up$, denoted by $\soc_{[\redez]}(\up)$, is a $[\redez]$-simple sequence $\us$ satisfying
$\us \preceq^\tb_{[\redez]} \up$ if such an $\us$ exists uniquely.
\end{definition}

\subsection{Socles, minimal pairs and folded distance polynomial}

\begin{proposition}[{\cite[Lemma 2.6]{BKM12}}]\label{pro: BKM minimal}
For $\ga \in \PR \setminus \Pi$ and any $\redez$ of $w_0$, a $[\redez]$-minimal sequence of $\ga$ is indeed a pair $(\al,\beta)$ for some $\al,\beta \in \PR$ such that
$\al+\beta = \ga$.
\end{proposition}

\begin{theorem} [\cite{OS16B,OS16C}] For any $[\mQ] \in \lf \mQ \rf$ and any pair $\up$, we have the followings$\colon$
\begin{enumerate}
\item[{\rm (1)}] $\soc_{[\mQ]}(\up)$ is well-defined.
\item[{\rm (2)}] $\dist_{[\mQ]}(\up) \le 2$. In particular, if $\dist_{[\mQ]}(\up)=2$, there exist a  unique $\um$ and a unique chain of length $3$
such that
$$\soc_{[\mQ]}(\up) \prec^{\tb}_{[\mQ]} \um \prec^{\tb}_{[\mQ]} \up.$$
\item[{\rm (3)}] If $\dist_{[\mQ]}(\up) \le 1$,
then
$\up$ is a $[\mQ]$-minimal pair of $\soc_{[\mQ]}(\up)$.
\end{enumerate}
\end{theorem}

For the involutions $\vee$ in \eqref{eq: B_n} and \eqref{eq: C_n}, we can identify $\overline{I}$, the orbit space of $\vee$,
with $\{ 1,2,\ldots,n\}$ and the order of $\vee$
is equal to $\overline{\mathsf{d}}\seteq 2$. The following propositions tell the characterization of the
positions of minimal pairs for $\ga \in \PR$ inside of $\WUp_{[\mQ]}$.

\begin{proposition}[{\cite[Proposition 7.8]{OS16B}}] Let us fix $[\mQ] \in \lf \mQ \rf$ of finite type $A_{2n-1}$.
For $\al,\beta,\ga \in \PR$ with $\widehat{\Omega}_{[\mQ]}(\al)=(i,p)$, $\widehat{\Omega}_{[\mQ]}(\beta)=(j,q)$
$\widehat{\Omega}_{[\mQ]}(\ga)=(k,r)$ and  $\al+\beta =\ga$, $(\al,\beta)$ is a $[\mQ]$-minimal pair of $\ga$
if and only if one of the following conditions holds$\colon$
\begin{eqnarray}&&
\left\{\hspace{1ex}\parbox{75ex}{
\begin{enumerate}
\item[{\rm (i)}] 
$\ell \seteq \max(i,j,k) \le n-1$, $i+j+k=2\ell$
and
$$ \left( \dfrac{q-r}{2},\dfrac{p-r}{2} \right) =
\begin{cases}
\big( -i,j \big), & \text{ if } \ell = k,\\
\big( i-(2n-1),j \big), & \text{ if } \ell = i,\\
\big( -i,2n-1-j  \big), & \text{ if } \ell = j.
\end{cases}
$$
\item[{\rm (ii)}] 
$s \seteq \min(i,j,k) \le n-1$, the others are the same as $n$ and
$$ (q-r,p-r) =
\begin{cases}
\big( -2(n-1-k)+1,2(n-1-k)-1 ), & \text{ if } s = k,\\
\big( -4i-4,2(n-1-i)-1  ), & \text{ if } s = i,\\
\big( -2(n-1-j)+1, 4j+4), & \text{ if } s = j.
\end{cases}
$$
\end{enumerate}
}\right. \label{eq: Dorey folded coordinate Bn}
\end{eqnarray}
\end{proposition}

\begin{proposition}[{\cite[Corollary 8.26]{OS16C}}] Let us fix $[\mQ] \in \lf \mQ \rf$ of finite type $D_{n+1}$.
For $\al,\beta,\ga \in \PR$ with $\widehat{\Omega}_{[\mQ]}(\al)=(i,p)$, $\widehat{\Omega}_{[\mQ]}(\beta)=(j,q)$
$\widehat{\Omega}_{[\mQ]}(\ga)=(k,r)$ and  $\al+\beta =\ga$, $(\al,\beta)$ is a $[\mQ]$-minimal pair of $\ga$
if and only if one of the following conditions holds$\colon$
\begin{eqnarray}&&
\left\{\hspace{1ex}\parbox{75ex}{
 $\ell \seteq \max(i,j,k) \le n$, $i+j+k=2\ell$
and
$$ \left( q-r,p-r \right) =
\begin{cases}
\big( -i,j \big), & \text{ if } \ell = k,\\
\big( i-(2n+2),j \big), & \text{ if } \ell = i,\\
\big( -i,2n+2-j  \big), & \text{ if } \ell = j.
\end{cases}
$$
}\right. \label{eq: Dorey folded coordinate Cn}
\end{eqnarray}
\end{proposition}

\begin{definition}[{ \cite[Definition 8.7]{OS16B}}]
For a folded AR-quiver $\WUp_{[\mQ]}$, indices $\ov{k},\ov{l} \in \ov{I}$ and an integer $t \in \Z_{\ge 1}$,
we define the subset $\Phi_{[\mQ]}(\ov{k},\ov{l})[t]$ of $\PR \times \PR$ as follows:

A pair $(\alpha,\beta)$ is contained in $\Phi_{[\mQ]}(\ov{k},\ov{l})[t]$ if $\alpha \prec_{[\mQ]} \beta$ or $\beta \prec_{[\mQ]} \al$ and
$$\{ \widehat{\phi}_{[\mQ]}(\al),\widehat{\phi}_{[\mQ]}(\beta) \}=\{ (\ov{k},a), (\ov{l},b)\} \quad \text{ such that } \quad |a-b|=t.$$
\end{definition}

\begin{proposition} [\cite{OS16B,OS16C}] For any $(\al^{(1)},\beta^{(1)}), \ (\al^{(2)},\beta^{(2)})  \in \Phi_{[\mQ]}(\ov{k},\ov{l})[t]$, we have
$$\dist_{[\mQ]}(\al^{(1)},\beta^{(1)})=\dist_{[\mQ]}(\al^{(2)},\beta^{(2)}).$$
Thus the notion
$$ o^{[\mQ]}_t(\ov{k},\ov{l}) \seteq \dist_{[\mQ]}(\al,\beta)  \quad \text{ for any } (\al,\beta) \in \Phi_{[\mQ]}(\ov{k},\ov{l})[t]$$
is well-defined.
\end{proposition}

\begin{definition}[{\cite[Definition 8.9]{OS16B}}] \label{def: Dist poly Q}
For $\ov{k},\ov{l} \in \ov{I}$ and a folded AR-quiver $\WUp_{[\mQ]}$, we define a polynomial $\widehat{D}^{[\mQ]}_{\ov{k},\ov{l}}(z) \in \ko[z]$ as follows:
Let $q$ be an indeterminate, $q_s^{\ov{\mathsf{d}}}=q_s^2=q$ and
$\mathtt{o}^{[\mQ]}_t(\ov{k},\ov{l}) \seteq \lceil o^{[\mQ]}_t(\ov{k},\ov{l}) / \ov{\mathsf{d}} \rceil $.
\begin{enumerate}
\item[{\rm (i)}] When $[\mQ]$ is of type $A_{2n-1}$,
$\widehat{D}^{[\mQ]}_{\ov{k},\ov{l}}(z) \seteq \displaystyle\prod_{ t \in \Z_{\ge 0} } (z- (-1)^{\ov{k}+\ov{l}}(q_s)^{t})^{\mathtt{o}^{[\mQ]}_t(\ov{k},\ov{l})}.$
\item[{\rm (ii)}] When $[\mQ]$ is of type $D_{n+1}$,
$\widehat{D}^{[\mQ]}_{\ov{k},\ov{l}}(z) \seteq \displaystyle\prod_{ t \in \Z_{\ge 0} } (z- (-q_s)^{t})^{\mathtt{o}^{[\mQ]}_t(\ov{k},\ov{l})}.$
\end{enumerate}
\end{definition}

\begin{proposition}[{\cite{OS16B,OS16C}}] \label{prop: DQ DQ'}
For $\ov{k},\ov{l} \in \ov{I}$ and any twisted adapted classes $[\mQ]$ and $[\mQ']$ in $\lf \mQ \rf$, we have
$$\widehat{D}^{[\mQ]}_{\ov{k},\ov{l}}(z)=\widehat{D}^{[\mQ']}_{\ov{k},\ov{l}}(z).$$
\end{proposition}

From the above proposition, we can define $\widehat{D}_{\ov{k},\ov{l}}(z)$ for $\lf \mQ \rf$ in a natural way
and call it {\em the folded distance polynomial} at $\ov{k}$ and $\ov{l}$.

\section{Quantum affine algebras, denominator formulas and Dorey's rule}

\subsection{Quantum affine algebras}
Let $\cmA$ be a generalized Cartan matrix of affine type; i.e., $\cmA$ is positive semi-definite of corank $1$.
We choose $0 \in I\seteq \{ 0,1,\ldots,n\}$ as the leftmost vertices in the tables in \cite[pages 54, 55]{Kac} except $A^{(2)}_{2n}$-case in which we take
the longest simple root as $\al_0$. We set $I_0 \seteq I \setminus \{ 0 \}$. We denote by $\delta \seteq  \sum_{i \in I} d_i \al_i$
the imaginary root and by $c=\sum_{i \in I} c_ih_i$ the center.
We have $d_0=1$.

For an affine Cartan datum $(\cmA,\wl,\Pi,\wl^{\vee},\Pi^{\vee})$, we denote by $\g$ the affine Kac-Moody algebra,
$\g_0$ the subalgebra generated by $\{ e_i,f_i,h_i \ | \ i \in I_0\}$,
by $U_q(\g)$
and $U_q(\g_0)$  the corresponding quantum groups.
We denote by
$U_q'(\g)$ the subalgebra of $U_q(\g)$ generated by $\{ e_i, f_i, q^{\pm h_i} \ | \ i \in I\}$. 
We mainly deal with
$U_q'(\g)$ which is called the {\em quantum affine algebra}.

We denote by $\mathcal{C}_\g$ the category of finite-dimensional integrable $U_q'(\g)$-modules. For the rest of this paper,
we take
the algebraic closure of $\C(q)$ in $\displaystyle\cup_{m >0} \C((q^{1/m}))$
as the base field  $\ko$ of $U_q'(\g)$-modules.
A simple module $M$ in $\mathcal{C}_\g$ contains a non-zero vector $u$ of weight $\lambda\in \wl_\cl \seteq \wl /\Z\delta$ such that
\begin{itemize}
\item $\langle c,\lambda \rangle =0$ and $\langle h_i,\lambda \rangle \ge 0$ for all $i \in I_0$,
\item all the weights of $M$ are contained in $\lambda - \sum_{i \in I_0} \Z_{\ge 0} \cl(\alpha_i)$,
\end{itemize}
where $\cl \colon \wl \to \wl_\cl$. Such a $\lambda$ is unique and $u$ is unique up to a constant multiple.
We call $\lambda$ the {\em dominant extremal weight} of $M$ and $u$ the {\em dominant extremal weight vector} of $M$.

For $M \in \mathcal{C}_\g$ and $x\in\cor^\times$, let $M_x$ be the $U_q'(\g)$-module with the actions of $e_i$, $f_i$ replaced with $x^{\delta_{i0}}e_i$, $x^{-\delta_{i0}}f_i$,
respectively.

For each $i \in I_0$, we set
$$\varpi_i \seteq {\rm gcd}(c_0,c_i)^{-1}\cl(c_0\Lambda_i-c_i \Lambda_0) \in \wl_\cl.$$

Then there exists a unique simple $U_q'(\g)$-module $V(\varpi_i)$ in $\mathcal{C}_\g$ with its dominant extremal weight $\varpi_i$ and
its dominant extremal weight vector $u_{\varpi_i}$, called the {\em fundamental representation of weight $\varpi_i$}, satisfying
certain conditions (see \cite[\S 1.3]{AK} for more detail). Moreover, there exist the left dual $V(\varpi_i)^*$ and the right
dual ${}^*V(\varpi_i)$ of $V(\varpi_i)$ with the following
$U_q'(\g)$-homomorphisms
\begin{equation} \label{eq: dual affine}
 V(\varpi_i)^* \otimes V(\varpi_i)  \overset{{\rm tr}}{\longrightarrow} \ko \quad \text{ and } \quad
V(\varpi_i) \otimes {}^*V(\varpi_i)  \overset{{\rm
tr}}{\longrightarrow} \ko.
\end{equation}
We have
\begin{equation} \label{eq: p star}
V(\varpi_i)^*\simeq  V(\varpi_{i^*})_{(p^*)^{-1}}, \
{}^*V(\varpi_i)\simeq  V(\varpi_{i^*})_{p^*} \ \  \text{ with } \ \
p^* \seteq (-1)^{\langle \rho^\vee ,\delta
\rangle} q^{\lan c,\rho\ran}.
\end{equation}
Here $\rho$ is defined by  $\langle h_i,\rho \rangle=1$,
$\rho^\vee$ is defined by $\langle \rho^\vee,\alpha_i  \rangle=1$ and
$i^*$ is the involution of $I_0$ defined in \eqref{eq:invI}.

For $ k \in \Z$ and $V(\varpi_i)_x$, we denote by
\eqn
V(\varpi_i)^{k*}_x &\seteq &
\bc
V(\varpi_i)_x &\text{if $k=0$,}\\[1ex]
( \cdots ((V(\varpi_i)_x \underbrace{ )^*)^*
\cdots)^{*}}_{\text{$k$-times}}
&\text{if $k>0$,}\\[3ex]
\underbrace{{}^*( \cdots {}^*({}^*(}_{\text{$-k$-times}}V(\varpi_i)_x)) \cdots)
&\text{if $k<0$.}\ec
\eneqn

We say that a $U_q'(\g)$-module $M$ is {\em good} if it has a {\it
bar involution}, a crystal basis with {\em simple crystal graph},
and a {\em global basis} (see \cite{Kas02} for the precise definition).
For instance, $V(\varpi_i)$
is a good module for every $i \in I$. 

\subsection{Denominator formulas and folded distance polynomials}
For a good module $M$ and $N$, there exists a $U_q'(\g)$-homomorphism
$$ \Rnorm_{M,N}: M_{z_M} \otimes M_{z_N} \to \ko(z_M,z_N)  \otimes_{\ko[z_M^{\pm 1},z_N^{\pm 1}]} N_{z_N} \otimes M_{z_M} $$
such that
$$ \Rnorm_{M,N} \circ z_M = z_M \circ \Rnorm_{M,N}, \ \Rnorm_{M,N} \circ z_N = z_N \circ \Rnorm_{M,N} \text{ and }
\Rnorm_{M,N} (u_M \otimes u_N) = u_N \otimes u_M,$$
where $u_M$ (resp.\ $u_N$) is the dominant extremal weight vector of $M$ (resp. $N$).

The {\em denominator} $d_{M,N}$ of $\Rnorm_{M,N}$ is the unique non-zero monic polynomial $d(u) \in \ko[u]$ of the smallest degree such that
\begin{equation}\label{definition: dm,n}
d_{M,N}(z_N/z_M)\Rnorm_{M,N}(M_{z_M} \otimes N_{z_N}) \subset N_{z_N} \otimes M_{z_M}.
\end{equation}

\begin{theorem} [\cite{AK,Chari,Kas02}]  \label{Thm: basic properties} 
\begin{enumerate}
\item[{\rm (1)}] For good modules $M_1$ and $M_2$, the zeroes of $d_{M_1,M_2}(z)$ belong to
$\C[[q^{1/m}]]\;q^{1/m}$ for some $m\in\Z_{>0}$.
\item[{\rm (2)}] $ V(\varpi_i)_{a_i} \otimes  V(\varpi_j)_{a_j}$ is simple if and only if
$$ d_{i,j}(z)\seteq d_{V(\varpi_i),V(\varpi_j)}(z) $$ does not vanish at $z=a_i/a_j$ nor $a_j/a_i$.
\item[{\rm (3)}]
Let $M$ be a finite-dimensional simple integrable $U'_q(\g)$-module
$M$. Then, there exists a finite sequence $$\left( (i_1,a_1),\ldots,
(i_l,a_l)\right) \text{ in } (I_0\times \ko^\times)^l$$ such that
$d_{i_k,i_{k'}}(a_{k'}/a_k) \not=0$ for $1\le k<k'\le l$ and $M$ is
isomorphic to the head of
$\bigotimes_{i=1}^{l}V(\varpi_{i_k})_{a_k}$. Moreover, such a
sequence $\left((i_1,a_1),\ldots, (i_l,a_l)\right)$ is unique up to
permutation.
\item[{\rm (4)}] $d_{k,l}(z)=d_{l,k}(z)=d_{k^*,l^*}(z)=d_{l^*,k^*}(z)$ for $k,l \in I_0$.
\end{enumerate}
\end{theorem}

The denominator formulas between fundamental representations
are calculated in \cite{AK,DO94,KKK13B,Oh14R} for all classical quantum affine algebras (see \cite[Appendix A]{Oh14R}).
In this paper, we will focus on the denominator formulas for $U'_q(B_n^{(1)})$ and $U'_q(C_n^{(1)})$:

\begin{proposition}[\cite{AK,Oh14R}]\hfill
\begin{enumerate}
\item[{\rm (1)}]
$ d^{B_{n}^{(1)}}_{k,l}(z) =
\begin{cases}
\displaystyle \prod_{s=1}^{\min (k,l)} \big(z-(-q)^{|k-l|+2s}\big)\big(z+(-q)^{2n-k-l-1+2s}\big) & 1 \le k,l \le n-1, \\
\displaystyle  \prod_{s=1}^{k}\big(z-(-1)^{n+k}q_s^{2n-2k-1+4s}\big)  & 1 \le k \le n-1, \ l=n, \\
\displaystyle \prod_{s=1}^{n} \big(z-(q_s)^{4s-2}\big) &  k=l=n.
 \end{cases}$
\vs{1ex}
 \item[{\rm (2)}]  $ d^{C_{n}^{(1)}}_{k,l}(z) = \displaystyle \hspace{-2ex} \prod_{s=1}^{ \min(k,l,n-k,n-l)} \hspace{-3ex}
 \big(z-(-q_s)^{|k-l|+2s}\big)\prod_{i=1}^{ \min(k,l)} \big(z-(-q_s)^{2n+2-k-l+2s}\big)$.
\end{enumerate}
\end{proposition}

The following theorem tells that we can read $d^{B_{n}^{(1)}}_{k,l}(z)$ from any $\WUp_{[\mQ]}$ of type $A_{2n-1}$
and $d^{C_{n}^{(1)}}_{k,l}(z)$ from any $\WUp_{[\mQ]}$ of type $D_{n+1}$:

\begin{theorem}[ \cite{OS16B,OS16C}] \label{thm: den dist}
For any $k,l \in \overline{I}$, we have
\begin{enumerate}
\item[{\rm (1)}] $d^{B_{n}^{(1)}}_{k,l}(z)= \widehat{D}_{k,l}(z) \times (z-q^{2n-1})^{\delta_{k,l}}$ where $\lf \mQ \rf$
is of type $A_{2n-1}$,
\item[{\rm (2)}] $d^{C_{n}^{(1)}}_{k,l}(z)= \widehat{D}_{k,l}(z) \times (z-q^{n+1})^{\delta_{k,l}}$ where $\lf \mQ \rf$
is of type $D_{n+1}$.
\end{enumerate}
\end{theorem}

\subsection{Dorey's rule and minimal pairs}
The morphisms in $${\rm Hom}_{U_q'(\g)}\big( V(\varpi_i)_a
\tens V(\varpi_j)_b, V(\varpi_k)_c \big) \quad \text{ for } i,j,k\in I_0 \text{ and } a,b,c \in \ko^\times$$
are studied by \cite{CP96,KKK13B,Oh14R} and called {\em Dorey's type} morphisms. In \cite{Oh14A,Oh14D},
the condition of non-vanishing of
the above Hom space
are interpreted the positions of $\al,\beta,\ga \in \PR$ in $\Gamma_Q$
where $(\al,\beta)$ is a pair for $\ga$ and $\g$ is of type $A^{(1)}_{n}$ or $D^{(1)}_{n}$.

\begin{theorem}[{\cite[Theorem 8.1, Theorem 8.2]{CP96}}]
For $\g^{(1)}=B^{(1)}_{n}$ or $C^{(1)}_{n}$,
let $(i,x)$, $(j,y)$, $(k,z) \in I_0 \times \ko^\times$. Then
$$ {\rm Hom}_{U'_q(\g^{(1)})}\big( V(\varpi_{j})_y \otimes V(\varpi_{i})_x , V(\varpi_{k})_z  \big) \ne 0 $$
if and only if one of the following conditions holds$\colon$
\begin{enumerate}
\item[{\rm (1)}] When $\g^{(1)}=B^{(1)}_{n}$, the conditions are given as follows$\colon$
\begin{eqnarray}&&
\left\{\hspace{1ex}\parbox{75ex}{
\begin{enumerate}
\item[{\rm (i)}] $\ell \seteq \max(i,j,k) \le n-1$, $i+j+k=2\ell$ and
$$ \left( y/z,x/z \right) =
\begin{cases}
\big( (-1)^{j+k}q^{-i},(-1)^{i+k}q^{j} \big), & \text{ if } \ell = k,\\
\big( (-1)^{j+k}q^{i-(2n-1)},(-1)^{i+k}q^{j} \big), & \text{ if } \ell = i,\\
\big( (-1)^{j+k}q^{-i},(-1)^{i+k}q^{2n-1-j}  \big), & \text{ if } \ell = j.
\end{cases}
$$
\item[{\rm (ii)}] $s \seteq \min(i,j,k) \le n-1$, the others are the same as $n$ and  
$$ (y/z,x/z) =
\begin{cases}
\big( (-1)^{n+k}q_s^{-2(n-1-k)+1},(-1)^{n+1+k}q_s^{2(n-1-k)-1} ), & \text{ if } s = k,\\
\big(  q_s^{-4i-4},(-1)^{i+n} q_s^{2(n-1-i)-1}  ), & \text{ if } s = i,\\
\big( (-1)^{j+n}q_s^{-2(n-1-j)+1}, q_s^{4j+4}), & \text{ if } s = j.
\end{cases}
$$
\end{enumerate}
}\right. \label{eq: Dorey B}
\end{eqnarray}
\item[{\rm (2)}] When $\g^{(1)}=C^{(1)}_{n}$, the conditions are given as follows$\colon$
\begin{eqnarray}&&
\left\{\hspace{1ex}\parbox{75ex}{
$\ell \seteq \max(i,j,k) \le n$, $i+j+k=2\ell$ and
$$ \left( y/z,x/z \right) =
\begin{cases}
\big( (-q_s)^{-i},(-q_s)^{j} \big), & \text{ if } \ell = k,\\
\big( (-q_s)^{i-(2n+2)},(-q_s)^{j} \big), & \text{ if } \ell = i,\\
\big( (-q_s)^{-i},(-q_s)^{2n+2-j}  \big), & \text{ if } \ell = j.
\end{cases}
$$
}\right. \label{eq: Dorey C}
\end{eqnarray}
\end{enumerate}
\end{theorem}


\begin{definition} [\cite{OS16B,OS16C}] \label{def: VmQ(beta)}
For any $[\mQ] \in \lf \mQ \rf$
and any positive root $\beta\in \PR$
of type $A_{2n-1}$ or $D_{n+1}$, we set the $U_q'(\g^{(1)})$-module ($\g^{(1)}=B_n^{(1)}$ or $C_n^{(1)}$) $V_{\mQ}(\beta)$ defined as follows$\colon$
For $\widehat{\Omega}_{\mQ}(\beta)=(i,p)$, we define
\begin{align} \label{eq: V(beta)}
V_{\mQ}(\beta) \seteq  \begin{cases}
V(\varpi_{i})_{(-1)^i(q_s)^{p}} & \text{ if } \g^{(1)}=B_n^{(1)}, \\
V(\varpi_{i})_{(-q_s)^{p}} & \text{ if } \g^{(1)}=C_n^{(1)}.
\end{cases}
\end{align}
\end{definition}

\begin{theorem} [\cite{OS16B,OS16C}] \label{thm: Dorey}
Let $(i,x)$, $(j,y)$, $(k,z) \in I_0 \times \ko^\times$. Then
$$ {\rm Hom}_{U_q'(\g^{(1)})}\big( V(\varpi_{j})_y \otimes V(\varpi_{i})_x , V(\varpi_{k})_z  \big) \ne 0  \quad \text{ for } \g^{(1)}=B^{(1)}_{n} \ \text{ $($resp.\ $C^{(1)}_{n})$} $$
if and only if there exists a twisted adapted class $[\mQ]$ of type $A_{2n-1}$ $($resp.\ $D_{n+1})$ and $\al,\beta,\ga \in \Phi_{A_{2n-1}}^+$
$($resp.\ $\Phi_{D_{n+1}}^+)$ such that
\begin{enumerate}
\item[{\rm (1)}] $(\al,\beta)$ is a $[\mQ]$-minimal pair of $\ga$,
\item[{\rm (2)}] $V(\varpi_{j})_y  = V_{\mQ}(\beta)_a, \ V(\varpi_{i})_x  = V_{\mQ}(\al)_a, \ V(\varpi_{k})_z  = V_{\mQ}(\ga)_a$
for some $a \in \ko^\times$.
\ee
\end{theorem}

\section{Categorifications and Schur-Weyl dualities}

\subsection{Categorifications via modules over \KLR  algebras}
For a given symmetrizable Cartan datum
$(\cmA,\wl,\Pi,\wl^\vee,\Pi^\vee)$, we choose a polynomial $\Q_{ij}(u,v) \in \ko[u,v]$ for $i,j \in I$
which is of the form
\begin{equation} \label{eq:Q}
\Q_{ij}(u,v) =
\delta(i \ne j)\sum\limits_{ \substack{ (p,q)\in \Z^2_{\ge 0} \\
p(\al_i | \al_i) + q(\al_j | \al_j)=-2(\al_i|\al_j)}}
t_{i,j;p,q} u^p v^q
\end{equation}
with $t_{i,j;p,q}\in\ko$, $t_{i,j;p,q}=t_{j,i;q,p}$ and
$t_{i,j:-a_{ij},0} \in \ko^{\times}$. Thus we have
$\Q_{i,j}(u,v)=\Q_{j,i}(v,u)$.

For $n \in \Z_{\ge 0}$ and $\mathsf{b} \in \rl^+$ such that $\het(\mathsf{b}) = n$, we set
$$I^{\mathsf{b}} = \{\nu = (\nu_1, \ldots, \nu_n) \in I^{n} \ | \ \alpha_{\nu_1} + \cdots + \alpha_{\nu_n} = \beta \}.$$

For $\mathsf{b} \in \rl^+$, we denote by $R(\mathsf{b})$ the \KLR algebra at $\mathsf{b}$ associated
with $(\cmA,\wl,\Pi,\wl^\vee,\Pi^\vee)$ and
$(\Q_{i,j})_{i,j \in I}$. It is a $\Z$-graded $\ko$-algebra generated by
the generators $\{ e(\nu) \}_{\nu \in  I^{\mathsf{b}}}$, $ \{x_k \}_{1 \le
k \le \het(\mathsf{b})}$, $\{ \tau_m \}_{1 \le m < \het(\mathsf{b})}$ with the certain defining relations (see \cite[Definition 2.7]{Oh14A} for the relations).

Let $\Rep(R(\mathsf{b}))$ be the category consisting of finite-dimensional graded $R(\mathsf{b})$-modules and $[\Rep(R(\mathsf{b}))]$ be the Grothendieck group of $\Rep(R(\mathsf{b}))$.
Then $[\Rep(R(\mathsf{b}))]$ has a natural $\Z[q^{\pm 1}]$-module structure induced by the grading shift. In this paper,
we often ignore grading shifts.

For $M \in \Rep(R(\mathsf{a}))$ and $N \in \Rep(R(\mathsf{b}))$, we denote by $M \conv N$ the
{\em convolution product} of $M$ and $N$. Then
 $\Rep(R)\seteq \soplus_{\mathsf{b} \in \rl^+} \Rep(R(\mathsf{b}))$
has a monoidal category structure by the convolution product and its
Grothendieck group $[\Rep(R)]$ has a natural $\Z[q^{\pm1}]$-algebra
structure induced by the convolution product $\conv$ and the grading
shift functor $q$.

For $M \in \Rep(\mathsf{b})$ and $M_k \in \Rep(\mathsf{b}_k)$ $(1 \le k \le n)$, we denote by
$$ M^{\conv 0} \seteq \ko, \quad M^{\circ r} = \overbrace{ M \conv \cdots \conv M }^r, \quad \dct{k=1}{n} M_k = M_1 \conv \cdots \conv M_n.$$

We say that a simple $R(\mathsf{b})$-module $M$ is {\em real} if $M\conv M$ is simple.

\begin{theorem}[{\cite{KL09, R08}}] \label{Thm:categorification}
For a given symmetrizable Cartan datum, Let $U^-_{\A}(\mathsf{g})^\vee$ $(\A=\Z[q^{\pm 1}])$ the dual of the integral form
of the negative part of quantum groups $U_q(\mathsf{g})$ and let
$R$ be the \KLR algebra related to the datum. Then we have
\begin{align}
U^-_{\A}(\mathsf{g})^{\vee} \simeq [\Rep(R)].
\label{eq:KLRU}
\end{align}
\end{theorem}

\begin{definition} We say that the \KLR algebra $R$ is {\em symmetric} if $\cmA$ is symmetric and
$\Q_{ij}(u,v)$ is a polynomial in $u-v$ for all $i,j\in I$.
\end{definition}

\begin{theorem} \cite{R11,VV09} \label{thm:categorification 2}
Assume that the \KLR algebra $R$ is symmetric and
the base field $\ko$ is of characteristic zero. Then under the isomorphism \eqref{eq:KLRU}
in {\rm Theorem \ref{Thm:categorification}},
the upper global basis of $U^-_{\A}(\mathsf{g})^{\vee}$ corresponds to
the set of the isomorphism classes of {\em self-dual} simple $R$-modules.
\end{theorem}

\begin{theorem} [\cite{Kato12,Mc12,Oh15E}]
For a finite-dimensional simple Lie algebra $\fg$, the dual PBW-basis of $U^-_\A(\fg)^\vee$ associated to $[\redez]$ is categorified in the following sense{\rm:}
for each $\beta \in \PR$, there exists a simple $R(\beta)$-module $S_{[\redez]}(\beta)$ such that
\begin{enumerate}
\item[{\rm (1)}] $S_{[\redez]}(\beta)^{\conv r}$ is simple for any $r \in \Z_{\ge 0}$,
\item[{\rm (2)}] for each $\um \in \Z_{\ge 0}^\N$, set $\Stom \seteq  S_{[\redez]}(\beta_1)^{\conv m_1}\conv\cdots \conv S_{[\redez]}(\beta_\N)^{\conv m_\N}$.
Then the set $\{ \Stom  \ | \ \um \in \Z_{\ge 0}^\N\}$ corresponds to
the dual PBW-basis under the isomorphism in \eqref{eq:KLRU} \ro up to a grading shifts\rof.
\end{enumerate}
\end{theorem}

\subsection{Categorifications via modules over untwisted quantum affine algebras}

\begin{definition}[ \cite{HL10,HL11,KKKOIV}] \label{def: VQ(beta)}
Fix any $[Q] \in \lf Q \rf$ of finite type $A_n$ or $D_n$,
and any positive root $\beta\in \PR$ with $\Omega_{Q}(\beta)=(i,p)$.
\begin{enumerate}
\item[{\rm (i)}] We set the $U_q'(\g^{(1)})$-module ($\g^{(1)}=A_n^{(1)}$ , $D_n^{(1)}$) $V^{(1)}_{Q}(\beta)$ defined as follows$\colon$
$$V^{(1)}_{Q}(\beta) \seteq  V(\varpi_{i})_{(-q)^{p}}.$$
\item[{\rm (ii)}] We set the $U_q'(\g^{(2)})$-module ($\g^{(2)}=A_n^{(2)}$ , $D_n^{(2)}$) $V^{(2)}_{Q}(\beta)$ defined as follows$\colon$
$$V^{(2)}_{Q}(\beta) \seteq  V(\varpi_{i^\star})_{((-q)^{p})^\star},$$
where
\begin{align}
(i^\star,((-q)^{p})^\star) \seteq
\begin{cases}
(i,(-q)^{p}) & \text{ if } \g^{(2)}=A^{(2)}_n \text{ and } 1 \le i \le \left\lfloor \dfrac{n+1}{2} \right\rfloor, \\
(n+1-i,(-1)^n(-q)^{p}) & \text{ if } \g^{(2)}=A^{(2)}_n \text{ and } \left\lfloor \dfrac{n+1}{2} \right\rfloor \le i \le n , \\
(i,(\sqrt{-1})^{n-i}(-q)^{p}) & \text{ if } \g^{(2)}=D^{(2)}_n \text{ and } 1 \le i \le n-2, \\
(n-1,(-1)^i(-q)^{p})& \text{ if } \g^{(2)}=D^{(2)}_n \text{ and } n-1 \le i \le n.
\end{cases}
\end{align}
\end{enumerate}

\begin{enumerate}
\item[{\rm (1)}]
We define the smallest abelian full subcategory
$\mathcal{C}^{(\ii)}_Q$ $(\ii=1,2)$ of $\mathcal{C}_{\g^{(\ii)}}$ such that
\begin{itemize}
\item[({\rm a})] it is stable by taking subquotient, tensor product and extension,
\item[({\rm b})] it contains $V^{(\ii)}_Q(\beta)$ for all $\beta \in \PR$.
\end{itemize}
\item[{\rm (2)}] We define the smallest abelian full subcategory
$\mathcal{C}^{(\ii)}_\Z$ $(\ii=1,2)$ of $\mathcal{C}_{\g^{(\ii)}}$ such that
\begin{itemize}
\item[({\rm a})] it is stable by taking subquotient, tensor product and extension,
\item[({\rm b})] it contains $V^{(\ii)}_Q(\beta)^{k*}$ for all $\beta \in \PR$ and all $k \in \Z$.
\end{itemize}
\end{enumerate}
\end{definition}
Note that the definition $\mathcal{C}^{(\ii)}_{\Z}$ does not depend on the choice of $Q$ and its height function.

\begin{theorem} [\cite{HL11,KKKOIV}]\label{thm: HLU}
We have a ring isomorphism given as follows$\colon$ For any $Q$ and $Q'$,
\begin{align}
\left[\mathcal{C}^{(1)}_Q \right] \simeq U^-_{\A}(\fg)^{\vee}|_{q=1} \simeq \left[\mathcal{C}^{(2)}_{Q'} \right],
\label{eq:HLU}
\end{align}
where $\left[\mathcal{C}^{(\ii)}_Q \right]$ denotes the Grothendieck ring of $\mathcal{C}^{(\ii)}_Q$ $(\ii=1,2)$.
\end{theorem}

\begin{theorem}[\cite{HL11,KKKOIV,Oh15E}] \label{thm: HLB}
A dual PBW-basis associated to $[Q]$ and the upper global basis of
$U^-_{\A}(\fg)^{\vee}$
are categorified by the modules over $U_q'(\g^{(t)})$ in the following sense $(t=1,2)\colon$
\begin{enumerate}
\item[{\rm (1)}] The set of all simple modules in $\mathcal{C}^{(\ii)}_Q$ corresponds to
the upper global basis of $U^-_{\A}(\fg)^{\vee}|_{q=1}$.
\item[{\rm (2)}] For each $\um \in \Z_{\ge 0}^\N$, define the  $U_q'(\g)$-module
$\VtoQm$ by $
V_{Q}^{(\ii)}(\beta_1)^{\tens m_1}\tens\cdots \tens V_{Q}^{(\ii)}(\beta_\N)^{\tens m_\N}$.
Then the set $\{ \VtoQm  \ | \ \um \in \Z_{\ge 0}^\N\}$ corresponds to
the dual PBW-basis under the isomorphism in \eqref{eq:HLU}.
\end{enumerate}
\end{theorem}

We note that Hernandez-Leclerc (\cite{HL11})
gave $E_n^{(1)}$-analogues of the above theorems.

\subsection{Generalized quantum affine Schur-Weyl dualities} \label{subsec: SW datum} 
In this subsection, we briefly review the generalized quantum affine Schur-Weyl duality which was studied in \cite{KKK13A,KKK13B,KKKOIII,KKKOIV}.

Let $\mathcal{S}$ be an index set. A Schur-Weyl datum $\Xi$ is a quintuple $$(U_q'(\g),J,X,s,\{V_s\}_{s \in \mathcal{S}})$$ consisting of
{\rm (a)} a quantum affine algebra $U_q'(\g)$, {\rm (b)} an index set $J$, {\rm (c)} two maps $X:J \to \ko^\times, \ s: J \to \mathcal{S}$,
{\rm (d)} a family of good $U'_q(\g)$-modules $\{V_s\}$ indexed by $\mathcal{S}$.

\medskip

For a given $\Xi$, we define a quiver $\Gamma^{\Xi}=(\Gamma^{\Xi}_0,\Gamma^{\Xi}_1)$ in the following way$\colon$
{\rm (i)} $\Gamma^{\Xi}_0= J$,
{\rm (ii)} for $i,j \in J$, we assign $\mathtt{d}_{ij}$ many arrows from $i$ to $j$, where $\mathtt{d}_{ij}$ is the order of the zero of
$d_{V_{s(i)},V_{s(j)}}(z_2/z_1)$ at $X(j)/X(i)$. We call $\Gamma^{\Xi}$ the {\em Schur-Weyl quiver associated to $\Xi$}.

For a Schur-Weyl quiver $\Gamma^{\Xi}$, we have
\begin{itemize}
\item a symmetric Cartan matrix $\cmA^{\Xi}=(a^{\Xi}_{ij})_{i,j \in J}$ by
\begin{align} \label{eq: sym Cartan mtx}
a^{\Xi}_{ij} =  2 \quad \text{ if } i =j, \quad a^{\Xi}_{ij} = -\mathtt{d}_{ij}-\mathtt{d}_{ji} \quad \text{ if } i \ne j,
\end{align}
\item the set of polynomials $(\mathcal{Q}^{\Xi}_{i,j}(u,v))_{i,j \in J}$
$$
\mathcal{Q}^{\Xi}_{i,j}(u,v) =(u-v)^{\mathtt{d}_{ij}}(v-u)^{\mathtt{d}_{ji}} \quad \text{ if } i \ne j.
$$
\end{itemize}

We denote by $R^{\Xi}$ the symmetric quiver Hecke algebra associated with $(\mathcal{Q}^{\Xi}_{i,j}(u,v))$.

\begin{theorem} [\cite{KKK13A}] \label{thm:gQASW duality} For a given ${\Xi}$, there exists a functor
$$\F : \Rep(R^{\Xi}) \rightarrow \mathcal{C}_\g.$$
Moreover, $\F$ satisfies the following properties$\colon$
\begin{enumerate}
\item[{\rm (1)}] $\F$ is a tensor functor; that is, there exist $U_q'(\g)$-module isomorphisms
$$\F(R^{\Xi}(0)) \simeq \ko \quad \text{ and } \quad \F(M_1 \circ M_2) \simeq \F(M_1) \tens \F(M_2)$$
for any $M_1, M_2 \in \Rep(R^{\Xi})$.
\item[{\rm (2)}] If the underlying graph of $\Gamma^{\Xi}$ is a Dynkin diagram of finite type $ADE$, then $\F$ is exact
and $R^{\Xi}$ is isomorphic to the \KLR algebra associated to
$\fg$ of finite type $ADE$.
\end{enumerate}
\end{theorem}
We call the functor $\F$ the {\em generalized quantum affine Schur-Weyl duality functor}.

\begin{theorem} [\cite{KKK13B,KKKOIV}] \label{thm:gQASW duality 2}
Let $U_q'(\g^{(t)})$ be a quantum affine algebra of type $A^{(t)}_{n}$  $($resp.\ $D^{(\ii)}_{n})$
and let $Q$ be a Dynkin quiver of finite type $A_n$  $($resp.\ $D_n)$ for $\ii=1,2$.
Take $J$ and $\mathcal{S}$ as the set of simple roots $\Pi$ associated to $Q$.
We define two maps $$ s: \Pi \to \{ V(\varpi_i) \ | \ i \in I_0 \} \quad \text{ and } \quad X: \Pi \to  \ko^\times $$ as follows$\colon$
for $\al \in \Pi$ with $\Omega_Q(\al)=(i,p)$, we define
$$ s(\al)=\begin{cases} V(\varpi_i) & \text{if $\g^{(1)}=A^{(1)}_{n}$ or $D^{(1)}_{n}$,} \\
V(\varpi_{i^\star}) & \text{if $\g^{(2)}=A^{(2)}_{n}$ or $D^{(2)}_{n}$,}
\end{cases}
\qquad X(\al)=
\begin{cases} (-q)^p  & \text{if $\g^{(1)}=A^{(1)}_{n}$ or $D^{(1)}_{n}$,} \\
((-q)^p)^\star & \text{if $\g^{(2)}=A^{(2)}_{n}$ or $D^{(2)}_{n}$.} \end{cases}$$
Then we have the followings$\colon$
\begin{itemize}
\item[{\rm (1)}] The underlying graph of $\Gamma^{\Xi}$ coincides with the one of $Q$. Hence the functor
$$\F^{(\ii)}_Q: \Rep(R^{\Xi}) \rightarrow \mathcal{C}^{(\ii)}_Q \quad (\ii =1,2)$$
in {\rm Theorem \ref{thm:gQASW duality}} is exact.
\item[{\rm (2)}] The functor $\F^{(\ii)}_Q$ induces a bijection from
the set of the isomorphism classes of simple objects of
$\Rep(R^{\Xi})$ to that of $\mathcal{C}^{(\ii)}_Q$. In particular,
$\F^{(\ii)}_Q$ sends $S_Q(\beta) \seteq S_{[Q]}(\beta)$ to
$V^{(\ii)}_Q(\beta)$. Moreover, the induced bijection between the
set of the isomorphism classes of simple objects of
$\mathcal{C}^{(1)}_Q$ and that of $\mathcal{C}^{(2)}_Q$ preserves
the dimensions.
\item[{\rm (3)}] The functors $\F^{(1)}_Q$ and $\F^{(2)}_Q$ induce the ring isomorphisms in \eqref{eq:HLU}. 
\end{itemize}
\end{theorem}

\section{Isomorphism between Grothendieck rings}
In this section, we first introduce subcategories $\mC_{\mQ}$ and $\mC_\Z$ of
$\mathcal{C}_{B^{(1)}_{n}}$ or
$\mathcal{C}_{C^{(1)}_{n}}$.

\begin{definition}[\cite{OS16B,OS16C}] \label{def: CmQ} \hfill 
\begin{enumerate}
\item[{\rm (i)}]
Let us define $\mC_\mQ^{B^{(1)}_{n}}$ \ro resp.\ $\mC_\mQ^{C^{(1)}_{n}}$\rof\
as the smallest abelian full subcategory
of $\mathcal{C}_{B^{(1)}_{n}}$ \ro resp.\ $\mathcal{C}_{C^{(1)}_{n}}$\rof\ such that
\begin{itemize}
\item[({\rm a})] it is stable by taking 
subquotient, tensor product and extension,
\item[({\rm b})] it contains $V_\mQ(\beta)$ for all $\beta \in \PR_{A_{2n-1}}$ $($resp.\ $\PR_{D_{n+1}})$.
\end{itemize}
\item[{\rm (ii)}] Let us define $\mC_\Z^{B^{(1)}_{n}}$ \ro resp.\ $\mC_\Z^{B^{(1)}_{n}}$\rof\ as
the smallest abelian full subcategory
 of $\mathcal{C}_{B^{(1)}_{n}}$ \ro resp.\ $\mathcal{C}_{C^{(1)}_{n}})$\rof\ such that
\begin{itemize}
\item[({\rm a})] it is stable by taking 
subquotient, tensor product and extension,
\item[({\rm b})] it contains $V_\mQ(\beta)^{k*}$ for all $k\in\Z$ and all $\beta \in \PR_{A_{2n-1}}$ $($resp.\ $\PR_{D_{n+1}})$.
\end{itemize}
\end{enumerate}
We sometimes omit the superscript ${B^{(1)}_{n}}$ or $C^{(1)}_{n}$ if there is no afraid of confusion.
\end{definition}

Note that the definition $\mC_\Z$ does not depend on the choice of $[\mQ]$ in $\lf \mQ \rf$.

\begin{theorem} \label{thm: exact functor} \hfill
\begin{enumerate}
\item[{\rm (1)}] There exists an exact functor
$\F_\mQ:\Rep(R_{A_{2n-1}})\to \mC_\mQ\subset \mathcal{C}_{B^{(1)}_{n}}$ for any $[\mQ]$ of type $A_{2n-1}$.
\item[{\rm (2)}] There exists an exact functor $\F_\mQ:\Rep(R_{D_{n+1}})\to \mC_\mQ
\subset\mathcal{C}_{C^{(1)}_{n}}$ for any $[\mQ]$ of type $D_{n+1}$.
\end{enumerate}
\end{theorem}
\begin{proof}
(1) For the construction of a functor, we need to take a Schur-Weyl datum $\Xi$.
(i) Take $J$ and $\mathcal{S}$ as the set of simple roots $\Pi$ of $\PR_{A_{2n-1}}$.
(ii) Define two maps $$ s: \Pi \to \{ V(\varpi_i) \ | \ i \in I_0 \} \quad \text{ and } \quad X: \Pi \to  \ko^\times $$ as follows$\colon$
For $\al \in \Pi$ with $\widehat{\Omega}(\al)=(i,p)$, we define
$$ s(\al)= V(\varpi_i) \ \text{ and }   \ X(\al)= (-1)^iq_s^p.$$
Then we can conclude that the underlying graph of Schur-Weyl quiver $\Gamma^{\Xi}$ coincides with the Dynkin diagram $\Delta$ of $A_{2n-1}$, since
\begin{itemize}
\item $\dist_{[\mQ]}(\al_{i},\al_{j}) =\begin{cases} 1 & \text{ if $i$ and $j$ are adjacent in $\Delta$,} \\ 0 & \text{ otherwise}, \end{cases}$
\item $d^{B_{n}^{(1)}}_{k,l}(z)= \widehat{D}_{k,l}(z) \times (z-q^{2n-1})^{\delta_{k,l}}$ (Theorem \ref{thm: den dist}),
\item $d^{B_{n}^{(1)}}_{k,l}(z)$ has only roots of order $1$.
\end{itemize}
Thus our assertion follows from {\rm (b)} of Theorem \ref{thm:gQASW duality}.

(2) The assertion can be proved by the same argument of (1) with the two maps $s$ and $X$ given as follows$\colon$
For $\al \in \Pi$ with $\widehat{\Omega}(\al)=(i,p)$, we define
$$ s(\al)= V(\varpi_i) \ \text{ and }   \ X(\al)= (-q_s)^p.$$
\end{proof}

\begin{theorem} \label{thm: FQ2 SQ}
For any $[\mQ]$ of $\lf \mQ \rf$
and $\ga \in \PR_{A_{2n-1}}$ $($resp.\ $\ga \in \PR_{D_{n+1}})$, we have
$$\F_\mQ(S_\mQ(\ga)) \simeq V_\mQ(\ga).$$
\end{theorem}

\begin{proof}
We shall prove our assertion by an induction on $\het(\ga)$. For
$\ga$ with $\het(\ga)=1$, our assertion follows from \cite[Proposition 3.2.2]{KKK13A} .
Now we assume that $\het(\ga) \ge 2$. Note that there exists a minimal pair $(\al,\beta)$ of
$\ga$. By \cite[Theorem 3.1]{Mc12}, we have a six-term exact
sequence of $R(\ga)$-modules
$$ 0 \To[t] S_{\mQ}(\ga)
\Lto S_{\mQ}(\al) \conv S_{\mQ}(\beta) \overset{r}{\Lto}
S_{\mQ}(\beta) \conv S_{\mQ}( \al ) \overset{s}{\Lto}
S_{\mQ}(\ga) \to 0.$$ Applying the functor $\F_\mQ$,
 we have an exact sequence of $U_q'(\g)$-modules by the induction
hypothesis
$$ 0 \To[\F_\mQ(t)] \F_\mQ(S_{\mQ}(\ga))
\to V_\mQ(\al) \tens V_\mQ(\beta) \overset{\F_\mQ(r)}{\Lto}
V_\mQ(\beta) \tens V_\mQ(\al) \overset{\F_\mQ(s)}{\Lto} \F_\mQ(S_{\mQ}(\ga)) \to 0.$$
On the other hand, Theorem \ref{thm: Dorey} tells that $V_{\mQ}(\beta) \tens V_{\mQ}(\al)$ is not simple.

We have then
$\F_\mQ(S_{\mQ}(\ga)) \ne 0$. Indeed, if it vanished, we would have
$$V_{\mQ}(\al) \tens V_{\mQ}(\beta) \simeq V_{\mQ}(\beta)
\tens V_\mQ(\al),$$
which implies that $V_\mQ(\al) \tens V_{\mQ}(\beta)$ is simple
by \cite[Corollary 3.16]{KKKO14S}.

Hence $\F_\mQ(S_{\mQ}(\ga))$ is the image of
a non-zero homomorphism
$$\F_\mQ(ts)\colon V_\mQ(\beta) \tens V_\mQ(\al)\to
 V_\mQ(\al) \tens V_\mQ(\beta).$$
Thus \cite{KKKO14S} and the quantum affine version of
 \cite[Proposition 4.5]{KKKO15B} imply that $\F_\mQ(S_{\mQ}(\beta))$ is the simple head of
$V_\mQ(\beta) \tens V_\mQ(\al)$ which coincides with $V_\mQ(\ga)$.
\end{proof}

\begin{lemma} \label{lem: k<l no poles} Let $\beta,\gamma\in \PR$.
If $ \Rnorm_{V_\mQ(\al),V_\mQ(\beta)}(z)$ has a pole at $z=1$, then
$\beta\prec_{[\mQ]} \al$.
\end{lemma}
\begin{proof}
The proof is almost identical with \cite[Lemma 4.2]{KKKOIV}.
\end{proof}

\begin{theorem} \label{thm: simples to simples}
The functor $\F_\mQ$ sends a simple module to a
simple module.
Moreover, the functor $\F_\mQ$ induces a bijection from
the set of simple modules in $\Rep(R)$ to
the set of simple modules in $\mC_\mQ$.
\end{theorem}

\begin{proof}
With Theorem \ref{thm: den dist} and Theorem \ref{thm: Dorey}, we can apply the same argument of \cite[Theorem 4.5]{KKKOIV}.
\end{proof}

Thus, for any $[Q]$, $[Q']$ and $[\mQ]$, we have the following diagrams$\colon$ 
\begin{align} \label{eq: dias}
\ba{c}\scalebox{0.9}{\xymatrix@C=.5ex@R=5ex{ &&\hs{7ex}
\mC_\mQ \subset \mathcal{C}_{B^{(1)}_n}\ar@{<.>}[ddrr]\ar@{<.>}[ddll]\\
&& \Rep(R_{A_{2n-1}}) \ar[u]^-{\F_{\mQ}}\ar[drr]_-{\F^{(2)}_{Q'}}\ar[dll]^-{\F^{(1)}_{Q}} \\
\mathcal{C}_{A^{(1)}_{2n-1}} \hspace{-2ex}  \supset \mathcal{C}^{(1)}_{Q} \ar@{<.>}[rrrr] && && \hs{2ex}\mathcal{C}^{(2)}_{Q'} \hspace{-0.8ex}
\subset \hspace{-0.4ex}  \mathcal{C}_{A^{(2)}_{2n-1}},}
}
\scalebox{0.9}{\xymatrix@C=.5ex@R=5ex{ && \hs{7ex}
\mC_\mQ\subset \mathcal{C}_{C^{(1)}_n} \ar@{<.>}[ddrr]\ar@{<.>}[ddll]  \\
&&  \Rep(R_{D_{n+1}}) \ar[u]^-{\F_{\mQ}}\ar[drr]_-{\F^{(2)}_{Q'}}\ar[dll]^-{\F^{(1)}_{Q}} \\
\mathcal{C}_{D^{(1)}_{n+1}} \hspace{-1.6ex} \supset \mathcal{C}^{(1)}_{Q} \ar@{<.>}[rrrr] && &&
\hs{2ex}\mathcal{C}^{(2)}_{Q'} \hspace{-0.8ex}  \subset \hspace{-0.4ex}  \mathcal{C}_{D^{(2)}_{n+1}}}
}\ea
\end{align}

\begin{corollary} Let $\fg=A_{2n-1}$ or $D_{n+1}$.
For any $[\mQ] \in \lf \mQ \rf$,
there exists an isomorphism between $[\mC_\mQ]$ and $U^-_\A(\fg)^\vee|_{q=1}$ induced by $\F_\mQ$$\colon$
\begin{align} \label{eq: isomQ}
[\F_\mQ] : [\mC_\mQ] \simeq U^-_\A(\fg)^\vee|_{q=1}.
\end{align}
Hence, we have isomorphisms which refine \eqref{eq:HLU}$\colon$
\begin{align} \label{eq: dias2}
\raisebox{2.8em}{\xymatrix@C=6ex@R=4ex{ & [\mC_\mQ] \ar@{<->}[ddr]^{\simeq}\ar@{<->}[ddl]_{\simeq}  \\
&  U^-_\A(\fg)^\vee|_{q=1} \ar@{<->}[u]^{\simeq}\ar@{<->}[dr]_{\simeq}\ar@{<->}[dl]^{\simeq} \\
[\mathcal{C}^{(1)}_{Q}] \ar@{<->}[rr]_{\simeq} &&  [\mathcal{C}^{(2)}_{Q'}]}}
\end{align}
\end{corollary}

Now we have a $[\mQ]$-analogue of Theorem \ref{thm: HLB}:

\begin{corollary} \label{cor: PBW upper global}
A dual PBW-basis associated to $[\mQ]$ and the upper global basis of
$U^-_{\A}(\fg)^{\vee}$
are categorified by the modules over $U_q'(\g)$ in the following sense$\colon$
\begin{enumerate}[{\rm (1)}]
\item The set of all simple modules in $\mC_\mQ$ corresponds to
the upper global basis of $U^-_{\A}(\fg)^{\vee}|_{q=1}$.
\item The set $\{ \VtomQm  \ | \ \um \in \Z_{\ge 0}^\N\}$ corresponds to
the dual PBW-basis under the isomorphism in \eqref{eq: isomQ}.
Here $\VtomQm \seteq  V_{\mQ}(\beta_1)^{\conv m_1}\conv\cdots \conv V_{\mQ}(\beta_\N)^{\conv m_\N}$.
\end{enumerate}
\end{corollary}

\begin{corollary}
For any $[Q]$ and $[\mQ]$ of type $A_{2n-1}$ $($resp.\ $D_{n+1})$, the ring isomorphism
$$  \phi^{(\ii)}_{Q,\mQ} \seteq  \F_\mQ  \circ (\F_Q^{(\ii)})^{-1} : [\mathcal{C}^{(\ii)}_Q]  \isoto{}  [\mC_\mQ] \quad (\ii=1,2)$$
sends simples to simples, bijectively.
\end{corollary}

\begin{corollary}
For any $[Q]$ and $[\mQ]$ of type $A_{2n-1}$ $ ($resp.\ $D_{n+1})$,
the ring isomorphism $\phi^{(\ii)}_{Q,\mQ}$ sends $[V_{Q}(\al_i)]$ to
$[V_{\mQ}(\al_i)]$ for each simple root $\al_i$.
\end{corollary}

\begin{proof}
For any $i \in I$, the $1$-dimensional module $L(i)$
is the unique simple module over $R(\al_i)$. Thus our assertion follows from \cite[Proposition 3.2.2]{KKK13A}.
\end{proof}

Now we have a conjecture which can be understood as a Langlands analogue of \cite[Conjecture 4.7]{KKKOIV}$\colon$
\begin{conjecture} \label{conj}
The functor $\F_{\mQ} \colon \Rep(R) \to \mC_\mQ$ is an equivalence of categories.
\end{conjecture}

\section{Simple head and socle} \label{sec: Simple head and socle}
In this section, we study the simple head and socle of $S_\mQ(\al) \conv S_\mQ(\beta)$ and
$V_\mQ(\al) \conv V_\mQ(\beta)$ which has been studied in many context (see \cite{KKKO14S,L03,Oh15E}).
\begin{theorem} \label{thm: simple simple}
A pair $\up=(\al,\beta)$
is $[\redez]$-simple if and only if $S_\mQ(\al) \conv S_\mQ(\beta)$ and
$V_\mQ(\al) \conv V_\mQ(\beta)$ are simple.
\end{theorem}

\begin{proof}
Only if part is an immediate consequence of \cite[Theorem 5.10]{Oh15E} (see also \cite[Theorem 3.1]{Mc12}).
Assume that $\dist_{[\mQ]}(\up)>0$. Then
there exists $\um$ such that $\um \prec^\tb_{[\mQ]} \up$ and there exists no $\um'$ such that
$$\um \prec^\tb_{[\mQ]} \um' \prec^\tb_{[\mQ]} \up.$$
Furthermore, by \cite[Theorem 6.16]{OS16B} and \cite[Theorem 8.12]{OS16C}, $\um$ satisfies one of the following conditions$\colon$
\begin{enumerate}
\item[{\rm (1)}] if $\um=\al+\beta$, $\up$ is a minimal pair of $\al+\beta$.
\item[{\rm (2)}] if (a) $\al+\beta \not \in \PR$, (b) $\dist_{[\mQ]}(\al,\beta)=2$ and (c) $[\mQ]$ is of type $A_{2n-1}$, $\um$ is a triple $(\mu,\nu,\eta)$ such that
\begin{itemize}
\item[{\rm (i)}] $\mu+\nu \in \PR$, $(\mu,\nu)$ is a $[\mQ]$-minimal pair of $\mu+\nu$ and $\al-\mu,\beta-\nu \in \PR$,
\item[{\rm (ii)}] $\eta$ is not comparable to $\mu$ and $\nu$ with respect to $\prec_{[\mQ]}$,
\item[{\rm (iii)}] $\eta=(\al-\mu)+(\beta-\nu)$ and $((\al-\mu),(\beta-\nu))$ is a $[\mQ]$-minimal pair for $\eta$,
\item[{\rm (iv)}] $(\al-\mu,\mu)$, $(\nu,\beta-\nu)$  are $[\mQ]$-minimal pairs for $\al$ and $\beta$ respectively,
\end{itemize}
\item[{\rm (3)}] if $\al+\beta \not \in \PR$ and it does not satisfy one of (b) and (c) in {\rm (2)}, then $\um$ is a pair $(\al',\beta')$ and either
{\rm (i)} $\al'-\al, \beta-\beta' \in \PR$ or {\rm (ii)} $\al-\al', \beta'-\beta \in \PR$
such that {\rm (i$^*$)} $(\al'-\al,\al)$ is a minimal pair for $\al$ or {\rm (ii$^*$)} $(\beta'-\beta,\beta)$ is a minimal pair for $\beta$.
\end{enumerate}
(see \cite[Remark 6.23]{OS16B} and \cite[Remark 8.19]{OS16C} also).
Thus our assertion for $\um=\al+\beta$ holds by \cite[Theorem 3.1]{Mc12}.
For the case when $\um$ is a pair, we have a non-zero composition of homomorphisms
\begin{itemize}
\item[{\rm (i)}] $S_\mQ(\al')\conv S_\mQ(\beta') \rightarrowtail S_\mQ(\al)\conv S_\mQ(\al'-\al) \conv  S_\mQ(\beta') \twoheadrightarrow S_\mQ(\al)\conv S_\mQ(\beta)$ or
\item[{\rm (ii)}] $S_\mQ(\al')\conv S_\mQ(\beta') \rightarrowtail S_\mQ(\al')\conv S_\mQ(\beta'-\beta) \conv  S_\mQ(\beta) \twoheadrightarrow S_\mQ(\al)\conv S_\mQ(\beta)$,
\end{itemize}
by \cite[Corollary 3.11]{KKKO14S}.
For the case when $\um$ is a triple, we have a non-zero composition
\begin{align*}
& S_\mQ(\mu)  \conv S_\mQ(\nu) \conv S_\mQ(\eta) \simeq S_\mQ(\mu) \conv S_\mQ(\eta) \conv S_\mQ(\nu) \\
& \hspace{5ex} \rightarrowtail  S_\mQ(\mu) \conv S_\mQ(\al-\mu) \conv S_\mQ(\beta-\nu) \conv  S_\mQ(\nu)
   \twoheadrightarrow S_\mQ(\mu) \conv S_\mQ(\al-\mu)  \conv  S_\mQ(\beta).
\end{align*}
and hence a desired non-zero composition
\begin{align*}
& S_\mQ(\mu)  \conv S_\mQ(\nu) \conv S_\mQ(\eta) \to S_\mQ(\mu) \conv  S_\mQ(\al-\mu)  \conv  S_\mQ(\beta) \twoheadrightarrow S_\mQ(\al)  \conv  S_\mQ(\beta)
\end{align*}
by \cite[Corollary 3.11]{KKKO14S}.
Hence our assertion follows from the fact that
the heads of $S_\mQ(\um)$ and $S_\mQ(\al)\conv S_\mQ(\beta)$ are distinct. Our assertion for
$V_\mQ(\al) \conv V_\mQ(\beta)$  can be obtained by applying the functor $\F_\mQ$.
\end{proof}

\begin{corollary} \label{cor: simple simple}
For $\um \in \Z_{\ge 0}^\N$,
$$S_{\mQ}(\um) \text{ and } V_{\mQ}(\um) \text{ are simple if and only if $\um$ is $[\mQ]$-simple.} $$
\end{corollary}

\begin{proof}
It is an immediate consequence of Theorem \ref{thm: simple simple}.
\end{proof}

\begin{lemma} \label{lem:simplepole}
Let $U_q'(\g)$ be a quantum affine algebra and let $V$ and $W$ be  good $U_q'(\g)$-modules.
If the normalized $R$-matrix $\Rnorm_{V,W}(z)$ has a simple pole at $z=a$ for some $a \in \ko^\times$, then we have
$${\rm Im} \big( (z-a) \Rnorm_{V,W} |_{z=a} \big) = {\rm Ker} \big( \Rnorm_{W,V} |_{z=a} \big). $$
Moreover, the tensor product $V \otimes W_a$  is of length $2$.
\end{lemma}

\begin{proof}
The first assertion follows from the fact$\colon$
\begin{eqnarray*}&&
\parbox{80ex}{
Let  $A(z)$ and $B(z)$ be  $n\times n$-matrices with entries in rational functions in $z$.
Assume that  $A(z)$ and $B(z)$ have no poles at $z=a$.
If $A(z)B(z) =(z-a){\rm id}$, then we have  ${\rm Im} A(a) ={\rm Ker} B(a)$.}
\end{eqnarray*}

Recall that ${\rm Ker} \big( \Rnorm_{W,V} |_{z=a} \big)$ is simple.
By \cite[Theorem 3.2]{KKKO14S}, we know that ${\rm Im} \big( (z-a) \Rnorm_{V,W} |_{z=a} \big)$ is also simple.
Hence we conclude that $V \otimes W_a$ is of composition length $2$ by the first assertion.
\end{proof}

\begin{theorem} For a pair $\up=(\al,\beta)$ with $\dist_{[\mQ]}(\up) >0$,
the composition length of $V_\mQ(\al) \conv V_\mQ(\beta)$ is $2$ and the composition series
of $V_\mQ(\al) \conv V_\mQ(\beta)$ consists of its distinct head and socle. In particular,
\begin{enumerate}
\item[{\rm (1)}] if $\dist_{[\mQ]}(\up)=1$, $\soc(V_\mQ(\al) \conv V_\mQ(\beta)) \simeq V_{\mQ}(\soc_{[\mQ]}(\al,\beta))$,
\item[{\rm (2)}] if $\dist_{[\mQ]}(\up)=2$, $\soc(V_\mQ(\al) \conv V_\mQ(\beta)) \simeq \hd(V_\mQ(\um))$ where
$\um$ is a unique sequence such that
$$ \soc_{[\mQ]}(\up) \prec^\tb_{[\mQ]} \um \prec^\tb_{[\mQ]} \up.$$
\end{enumerate}
The same assertions for $S_\mQ(\al) \conv S_\mQ(\beta)$ hold.
\end{theorem}

\begin{proof}
By Theorem \ref{thm: simple simple}, Proposition \ref{lem:simplepole} and the fact that $d^{B_n^{(1)}}_{k,l}(z)$ and $d^{C_n^{(1)}}_{k,l}(z)$
have only roots of order $1$, $V_\mQ(\al) \conv V_\mQ(\beta)$ has composition length $2$ if $\dist_{\mQ}(\al,\beta) \ne 0$. If
$\dist_{[\mQ]}(\up)=1$, then $\soc_\mQ(\up)$ is a unique sequence such that $\soc_\mQ(\up) \prec_{[\mQ]}^\tb \up$
and $V_\mQ(\soc_\mQ(\up))$ is simple. Thus, by the previous proof, there exist a non-zero homomorphism
$$V_\mQ(\soc_\mQ(\up)) \Lto V_\mQ(\up).$$ Hence $V_\mQ(\soc_\mQ(\up))\simeq \soc(V_\mQ(\up))$ by \cite[Theorem 3.2]{KKKO14S}. For $\dist_{\mQ}(\up)=2$,
we have a non-zero homomorphism
$$V_\mQ(\um) \Lto V_\mQ(\up)$$
where $\um$ is the unique sequence such that $\soc_\mQ(\up) \prec_{[\mQ]}^\tb \um \prec_{[\mQ]}^\tb \up$. Thus our last assertion
follows since
\begin{itemize}
\item $V_\mQ(\up)$ has composition length $2$,
\item $V_\mQ(\up)$ and $V_\mQ(\um)$ have distinct heads.
\end{itemize}
Our assertions for $S_\mQ(\up)$ can be obtained by applying the functor $\F_\mQ$.
\end{proof}

\section*{Appendix: Exceptional doubly laced type}

In this appendix, we discuss the exceptional doubly laced type analogue of our main results and give several conjectures on it. We first recall the $\vee$-foldable
cluster point $\lf \mQ \rf$ of $E_6$ associated with $\vee$ in \eqref{eq: F_4} and the twisted Coxeter element $s_1s_2s_6s_3$ \cite[Appendix]{OS16B}$\colon$
$$ \lf \mQ \rf = \lf \redez \rf \qquad \text{ where } \quad \redez= \prod_{k=0}^{8} (s_1s_2s_6s_3)^{k\vee}.  $$
Here,
\begin{align} \label{eq: vee def}
& (s_{j_1} \cdots s_{j_n})^\vee \seteq s_{j^\vee_1} \cdots s_{j^\vee_n} \text{ and }
(s_{j_1} \cdots s_{j_n})^{k \vee} \seteq  ( \cdots ((s_{j_1} \cdots s_{j_n} \underbrace{ )^\vee )^\vee \cdots )^\vee}_{ \text{ $k$-times} }.
\end{align}
Note that the number of distinct commutation classes in $\lf \mQ \rf$ is $32$ (\cite[Appendix]{OS16B}).

Now, we assign the coordinates of $\Upsilon_{[\redez]}$ in the following way (see also \cite[Appendix]{Oh15E})$\colon$
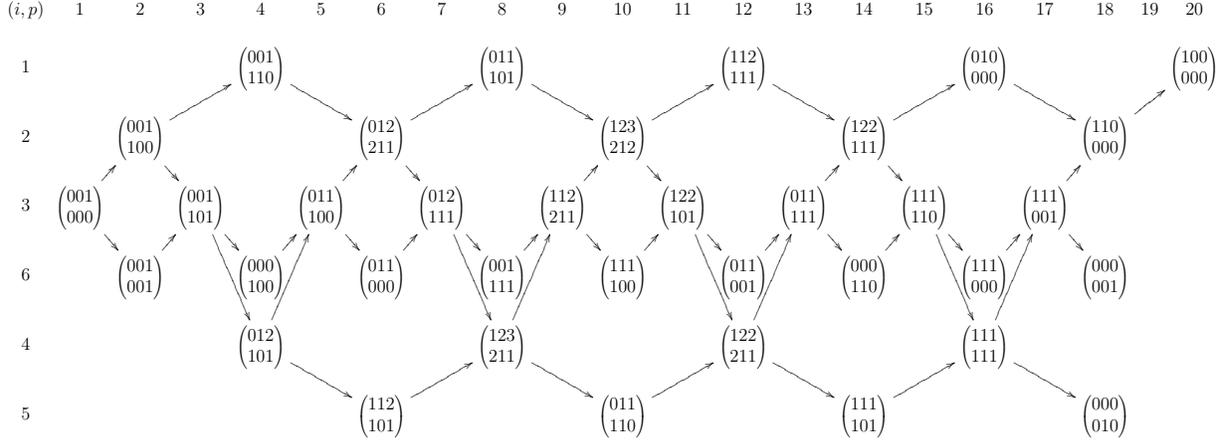
$$\scalebox{0.57}{\xymatrix@C=0.1ex@R=2ex{
(i,p)& 1&2& 3& 4& 5& 6& 7& 8& 9& 10& 11& 12& 13& 14&  15& 16& 17& 18& 19 & 20 \\
1 &&&& {\scriptstyle\prt{001}{110}} \ar@{->}[drr] &&&& {\scriptstyle\prt{011}{101}}\ar@{->}[drr] &&&& {\scriptstyle\prt{112}{111}}\ar@{->}[drr]
&&&& {\scriptstyle\prt{010}{000}} \ar@{->}[drr] &&&& {\scriptstyle\prt{100}{000}} \\
2 && {\scriptstyle\prt{001}{100}} \ar@{->}[urr]\ar@{->}[dr] &&&& {\scriptstyle\prt{012}{211}}\ar@{->}[urr]\ar@{->}[dr] &&&&{\scriptstyle\prt{123}{212}}
\ar@{->}[urr]\ar@{->}[dr]
&&&&{\scriptstyle\prt{122}{111}} \ar@{->}[urr]\ar@{->}[dr] &&&& {\scriptstyle\prt{110}{000}}\ar@{->}[urr]\\
3 & {\scriptstyle\prt{001}{000}}\ar@{->}[ur]\ar@{->}[dr] && {\scriptstyle\prt{001}{101}} \ar@{->}[dr]\ar@{->}[ddr] && {\scriptstyle\prt{011}{100}} \ar@{->}[ur]\ar@{->}[dr]
&& {\scriptstyle\prt{012}{111}} \ar@{->}[dr]\ar@{->}[ddr] && {\scriptstyle\prt{112}{211}}\ar@{->}[ur]\ar@{->}[dr] &&{\scriptstyle\prt{122}{101}}\ar@{->}[ddr]\ar@{->}[dr]
&&{\scriptstyle\prt{011}{111}}\ar@{->}[ur]\ar@{->}[dr] &&
{\scriptstyle\prt{111}{110}}\ar@{->}[dr]\ar@{->}[ddr] && {\scriptstyle\prt{111}{001}}\ar@{->}[ur]\ar@{->}[dr]\\
6 && {\scriptstyle\prt{001}{001}}\ar@{->}[ur] && {\scriptstyle\prt{000}{100}} \ar@{->}[ur] && {\scriptstyle\prt{011}{000}} \ar@{->}[ur] &&
{\scriptstyle\prt{001}{111}}\ar@{->}[ur] &&{\scriptstyle\prt{111}{100}}\ar@{->}[ur] && {\scriptstyle\prt{011}{001}} \ar@{->}[ur]&& {\scriptstyle\prt{000}{110}}
\ar@{->}[ur] &&{\scriptstyle\prt{111}{000}}\ar@{->}[ur]  &&{\scriptstyle\prt{000}{001}} \\
4 &&&& {\scriptstyle\prt{012}{101}}\ar@{->}[drr]\ar@{->}[uur] &&&& {\scriptstyle\prt{123}{211}}\ar@{->}[drr]\ar@{->}[uur]
&&&& {\scriptstyle\prt{122}{211}}\ar@{->}[drr]\ar@{->}[uur] &&&& {\scriptstyle\prt{111}{111}} \ar@{->}[drr]\ar@{->}[uur]\\
5 &&&&&& {\scriptstyle\prt{112}{101}}\ar@{->}[urr]  &&&&
{\scriptstyle\prt{011}{110}} \ar@{->}[urr] &&&&
{\scriptstyle\prt{111}{101}} \ar@{->}[urr]  &&&&
{\scriptstyle\prt{000}{010}} }}
$$
$\left({\scriptstyle\prt{a_1a_2a_3}{a_4a_5a_6}} \seteq \displaystyle\sum_{i=1}^6 a_i\al_i \right)$. Then the quiver $\Upsilon_{[\redez]}$ is {\em foldable}
in the sense that there exists no $(i,p)$ and $(j,q) \in \Upsilon_{[\redez]}$ such that $i^\vee=j$ and $p=q$. Hence we can fold in a canonical way$\colon$
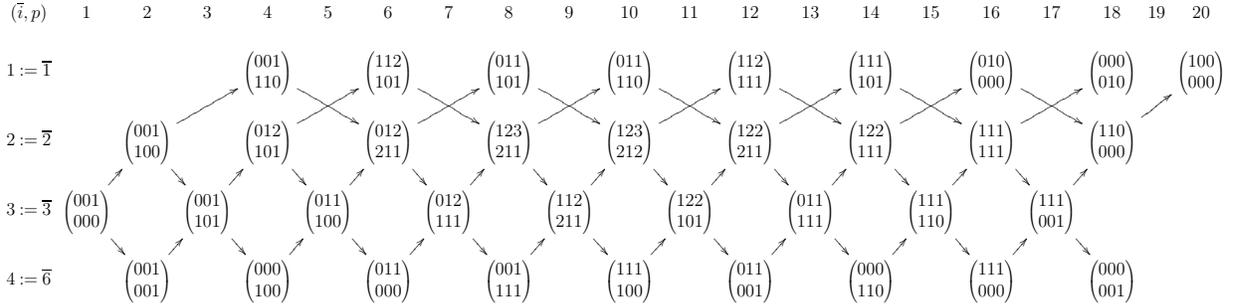
$$
\scalebox{0.57}{\xymatrix@C=0.1ex@R=2ex{
(\overline{i},p)& 1&2& 3& 4& 5& 6& 7& 8& 9& 10& 11& 12& 13& 14&  15& 16& 17& 18& 19 & 20 \\
1\seteq\overline{1} &&&& {\scriptstyle\prt{001}{110}} \ar@{->}[drr] &&{\scriptstyle\prt{112}{101}}\ar@{->}[drr]&&
{\scriptstyle\prt{011}{101}} \ar@{->}[drr] &&{\scriptstyle\prt{011}{110}} \ar@{->}[drr]&&
{\scriptstyle\prt{112}{111}}\ar@{->}[drr] &&{\scriptstyle\prt{111}{101}} \ar@{->}[drr]&&
{\scriptstyle\prt{010}{000}} \ar@{->}[drr] &&{\scriptstyle\prt{000}{010}}
&& {\scriptstyle\prt{100}{000}}\\
2\seteq\overline{2} && {\scriptstyle\prt{001}{100}} \ar@{->}[urr]\ar@{->}[dr] && {\scriptstyle\prt{012}{101}}\ar@{->}[urr]\ar@{->}[dr]&&
{\scriptstyle\prt{012}{211}}\ar@{->}[urr]\ar@{->}[dr]
&&{\scriptstyle\prt{123}{211}}\ar@{->}[urr]\ar@{->}[dr]&&{\scriptstyle\prt{123}{212}} \ar@{->}[urr]\ar@{->}[dr]
&&{\scriptstyle\prt{122}{211}}\ar@{->}[urr]\ar@{->}[dr]&&{\scriptstyle\prt{122}{111}} \ar@{->}[urr]\ar@{->}[dr]
&&{\scriptstyle\prt{111}{111}} \ar@{->}[urr]\ar@{->}[dr]&& {\scriptstyle\prt{110}{000}}\ar@{->}[urr] \\
3\seteq\overline{3} & {\scriptstyle\prt{001}{000}}\ar@{->}[ur]\ar@{->}[dr] &&
{\scriptstyle\prt{001}{101}} \ar@{->}[dr]\ar@{->}[ur] &&
{\scriptstyle\prt{011}{100}} \ar@{->}[ur]\ar@{->}[dr] &&
{\scriptstyle\prt{012}{111}} \ar@{->}[dr]\ar@{->}[ur] &&
{\scriptstyle\prt{112}{211}}\ar@{->}[ur]\ar@{->}[dr]
&&{\scriptstyle\prt{122}{101}}\ar@{->}[ur]\ar@{->}[dr]
&&{\scriptstyle\prt{011}{111}}\ar@{->}[ur]\ar@{->}[dr] &&
{\scriptstyle\prt{111}{110}}\ar@{->}[dr]\ar@{->}[ur] && {\scriptstyle\prt{111}{001}}\ar@{->}[ur]\ar@{->}[dr]\\
4\seteq\overline{6} && {\scriptstyle\prt{001}{001}}\ar@{->}[ur] &&
{\scriptstyle\prt{000}{100}} \ar@{->}[ur] &&
{\scriptstyle\prt{011}{000}} \ar@{->}[ur] &&
{\scriptstyle\prt{001}{111}}\ar@{->}[ur]
&&{\scriptstyle\prt{111}{100}}\ar@{->}[ur] &&
{\scriptstyle\prt{011}{001}} \ar@{->}[ur]&&
{\scriptstyle\prt{000}{110}}
\ar@{->}[ur] &&{\scriptstyle\prt{111}{000}}\ar@{->}[ur]  &&{\scriptstyle\prt{000}{001}} }}
$$
Thus $\widehat{\Upsilon}_{[\redez]}$ is well-defined.

Furthermore, the action of reflection maps on $\lf \mQ \rf$ is well-described by the datum of $F_4$ in the following sense (see also
\cite[Algorithm 7.6]{OS16B}, \cite[Algorithm 7.15]{OS16C})$\colon$
$$
\Delta : \xymatrix@R=0.5ex@C=4ex{
*{\circ}<3pt> \ar@{-}[r]_<{1 \ \ }  &*{\circ}<3pt>
\ar@{=>}[r]_<{2 \ \ } & *{\circ}<3pt> \ar@{-}[r]_<{3 \ \ } & *{\circ}<3pt> \ar@{-}[l]^<{ \ \ 4}} \text{ of type $F_4$}
$$
Let us denote by {\rm (i)} $\mathsf{D}={\rm diag}(d_{i} \ | \ 1 \le i \le 4 )$ the diagonal matrix
which diagonalizes the Cartan matrix $\cmA$ of type $F_{4}$, {\rm (ii)} $\overline{\mathsf{d}}={\rm lcm}(d_{i} \ | \ 1 \le i \le 4)=2$,
{\rm (iii)} $\al_i$ a sink of $\widehat{\Upsilon}_{[\redez]}$
and {\rm (iv)} $\mathsf{h}^\vee=9$ the dual Coxeter number of type $F_4$. Now the algorithm obtaining $\widehat{\Upsilon}_{[\redez]r_i}$ from
$\widehat{\Upsilon}_{[\redez]}$ can be described as follows$\colon$
\begin{enumerate}
\item[{\rm (A1)}] Remove the vertex $(i,p)$ corresponding $(\al_i)$ and arrows entering into $(i,p)$ in
$\widehat{\Upsilon}_{[\redez]}$.
\item[{\rm (A2)}] Add the vertex $(i,p-\overline{\mathsf{d}} \times \mathsf{h}^\vee)$ and
arrows to all vertices whose coordinates are $(j, p-\overline{\mathsf{d}} \times \mathsf{h}^\vee+\min(d_{i},d_{j})) \in \widehat{\Upsilon}_{[\redez]}$,
where $j$ is adjacent to $i$ in $\Delta$.
\item[{\rm (A3)}] Label the vertex $(i,p-\overline{\mathsf{d}} \times \mathsf{h}^\vee)$ with $\al_i$ and change the labels $\beta$ to
$s_i(\beta)$ for all $\beta \in \widehat{\Upsilon}_{[\redez]}
\setminus \{\al_i\}$.
\end{enumerate}

\medskip

For example $\widehat{\Upsilon}_{[\redez]r_1}$ can be depicted as follows$\colon$
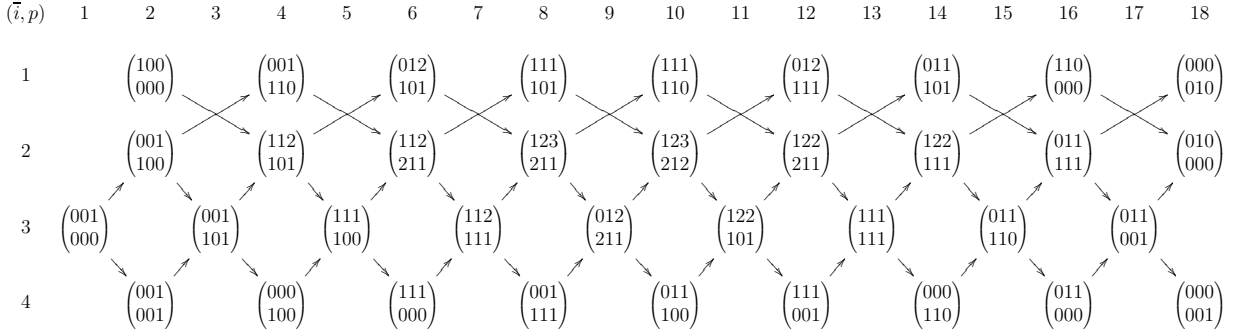
$$
\scalebox{0.62}{\xymatrix@C=0.1ex@R=2ex{
(\overline{i},p)& 1&2& 3& 4& 5& 6& 7& 8& 9& 10& 11& 12& 13& 14&  15& 16& 17& 18 \\
1 && {\scriptstyle\prt{100}{000}} \ar@{->}[drr] && {\scriptstyle\prt{001}{110}} \ar@{->}[drr] &&{\scriptstyle\prt{012}{101}}\ar@{->}[drr]&&
{\scriptstyle\prt{111}{101}} \ar@{->}[drr] &&{\scriptstyle\prt{111}{110}} \ar@{->}[drr]&&
{\scriptstyle\prt{012}{111}}\ar@{->}[drr] &&{\scriptstyle\prt{011}{101}} \ar@{->}[drr]&&
{\scriptstyle\prt{110}{000}} \ar@{->}[drr] &&{\scriptstyle\prt{000}{010}}
\\
2 && {\scriptstyle\prt{001}{100}} \ar@{->}[urr]\ar@{->}[dr] && {\scriptstyle\prt{112}{101}}\ar@{->}[urr]\ar@{->}[dr]&&
{\scriptstyle\prt{112}{211}}\ar@{->}[urr]\ar@{->}[dr]
&&{\scriptstyle\prt{123}{211}}\ar@{->}[urr]\ar@{->}[dr]&&{\scriptstyle\prt{123}{212}} \ar@{->}[urr]\ar@{->}[dr]
&&{\scriptstyle\prt{122}{211}}\ar@{->}[urr]\ar@{->}[dr]&&{\scriptstyle\prt{122}{111}} \ar@{->}[urr]\ar@{->}[dr]
&&{\scriptstyle\prt{011}{111}} \ar@{->}[urr]\ar@{->}[dr]&& {\scriptstyle\prt{010}{000}} \\
3 & {\scriptstyle\prt{001}{000}}\ar@{->}[ur]\ar@{->}[dr] &&
{\scriptstyle\prt{001}{101}} \ar@{->}[dr]\ar@{->}[ur] &&
{\scriptstyle\prt{111}{100}} \ar@{->}[ur]\ar@{->}[dr] &&
{\scriptstyle\prt{112}{111}} \ar@{->}[dr]\ar@{->}[ur] &&
{\scriptstyle\prt{012}{211}}\ar@{->}[ur]\ar@{->}[dr]
&&{\scriptstyle\prt{122}{101}}\ar@{->}[ur]\ar@{->}[dr]
&&{\scriptstyle\prt{111}{111}}\ar@{->}[ur]\ar@{->}[dr] &&
{\scriptstyle\prt{011}{110}}\ar@{->}[dr]\ar@{->}[ur] && {\scriptstyle\prt{011}{001}}\ar@{->}[ur]\ar@{->}[dr]\\
4 && {\scriptstyle\prt{001}{001}}\ar@{->}[ur] &&
{\scriptstyle\prt{000}{100}} \ar@{->}[ur] &&
{\scriptstyle\prt{111}{000}} \ar@{->}[ur] &&
{\scriptstyle\prt{001}{111}}\ar@{->}[ur]
&&{\scriptstyle\prt{011}{100}}\ar@{->}[ur] &&
{\scriptstyle\prt{111}{001}} \ar@{->}[ur]&&
{\scriptstyle\prt{000}{110}}
\ar@{->}[ur] &&{\scriptstyle\prt{011}{000}}\ar@{->}[ur]  &&{\scriptstyle\prt{000}{001}} }}
$$

\begin{remark} For every $\Upsilon_{[\mQ]}$ ($[\mQ] \in \lf \mQ \rf$) of $E_6$, {\em the twisted
additive property} observed in \cite{OS16B,OS16C} also holds.
\end{remark}

By applying the results in \cite{OS16B,OS16C}, one can check that the folded distance polynomials $D_{k,l}(z)$ are well-defined on $\lf \mQ \rf$
and have natural conjectural formulas for $d^{F^{(1)}_4}_{k,l}(z)$ as follows$\colon$ Set $q^2_s=q$.
\fontsize{9}{9}\selectfont
\begin{align*}
d^{F^{(1)}_4}_{1,1}(z)& =(z-(-q_s)^{4})(z-(-q_s)^{10})(z-(-q_s)^{12})(z-(-q_s)^{18}), \\
d^{F^{(1)}_4}_{1,2}(z)& =(z-(-q_s)^{6})(z-(-q_s)^{10})(z-(-q_s)^{12})(z-(-q_s)^{14})(z-(-q_s)^{16}), \\
d^{F^{(1)}_4}_{1,3}(z)& =(z-(-q_s)^{7})(z-(-q_s)^{9})(z-(-q_s)^{13})(z-(-q_s)^{15}), \\
d^{F^{(1)}_4}_{1,4}(z)& =(z-(-q_s)^{8})(z-(-q_s)^{14}), \\
d^{F^{(1)}_4}_{2,2}(z)& =(z-(-q_s)^{4})(z-(-q_s)^{6})(z-(-q_s)^{8})(z-(-q_s)^{10}) (z-(-q_s)^{12})
             (z-(-q_s)^{14})^2(z-(-q_s)^{16})(z-(-q_s)^{18}), \\
d^{F^{(1)}_4}_{2,3}(z)& =(z-(-q_s)^{5})(z-(-q_s)^{7})(z-(-q_s)^{9})(z-(-q_s)^{11})^2(z-(-q_s)^{13})
            (z-(-q_s)^{15})(z-(-q_s)^{17}), \\
d^{F^{(1)}_4}_{2,4}(z)& =(z-(-q_s)^{6})(z-(-q_s)^{10})(z-(-q_s)^{12})(z-(-q_s)^{16}), \\
d^{F^{(1)}_4}_{3,3}(z)& =(z-(-q_s)^{2})(z-(-q_s)^{6})(z-(-q_s)^{8})(z-(-q_s)^{10})(z-(-q_s)^{12})(z-(-q_s)^{16})(z-(-q_s)^{18}), \\
d^{F^{(1)}_4}_{3,4}(z)& =(z-(-q_s)^{3})(z-(-q_s)^{7})(z-(-q_s)^{11})(z-(-q_s)^{13})(z-(-q_s)^{17}),\\
d^{F^{(1)}_4}_{4,4}(z)& =(z-(-q_s)^{2})(z-(-q_s)^{8})(z-(-q_s)^{12})(z-(-q_s)^{18}).
\end{align*}
\fontsize{12}{12}\selectfont

Now we can define $V_{\mQ}(\beta)$, for each $\beta \in \PR_{E_6}$ and $[\mQ] \in \lf \mQ \rf$ naturally.

\begin{definition} \hfill
\begin{enumerate}
\item[{\rm (1)}]
We define the smallest abelian full subcategory
$\mC_\mQ$ inside $\mathcal{C}_{F^{(1)}_{4}}$ such that
\begin{itemize}
\item[({\rm a})] it is stable by taking subquotient, tensor product and extension,
\item[({\rm b})] it contains $V_\mQ(\beta)$ for all $\beta \in \PR_{E_{6}}$.
\end{itemize}
\item[{\rm (2)}] We define the smallest abelian full subcategory
$\mC_\Z$ inside $\mathcal{C}_{F^{(1)}_{4}}$ such that
\begin{itemize}
\item[({\rm a})] it is stable by taking subquotient, tensor product and extension,
\item[({\rm b})] it contains $V_\mQ(\beta)^{k*}$ for all $k\in\Z$ and all $\beta \in \PR_{E_{6}}$.
\end{itemize}
\end{enumerate}
\end{definition}

Recall the subcategories $\mathcal{C}^{(1)}_Q$ of $\mathcal{C}_{E^{(1)}_{6}}$ in \cite{HL11},
$\mathcal{C}^{(2)}_Q$ of $\mathcal{C}_{E^{(2)}_{6}}$ in \cite[Arxiv version 1]{Oh15E}. Now we can naturally expect that
all results in this paper can be extended to $\mathcal{C}^{(\ii)}_Q$ $(\ii=1,2)$, $\mC_\mQ$, $U^-_\A(E_6)^\vee$ and $V_\mQ(\um)$, etc.


\end{document}